\documentclass[10pt,a4paper,reqno,twoside]{amsart}
\usepackage{amsmath}
\usepackage{amssymb}
\usepackage{amsthm}
\usepackage{mathrsfs}
\usepackage{color}
\usepackage{xcolor}
\usepackage{amsmath}
\usepackage{mathrsfs}

\usepackage{amsfonts}
\usepackage{amsthm}
\usepackage{verbatim}
\usepackage{hyperref}
\hypersetup{hidelinks}

\newcommand{\R}{{\mathbb R}}

\newcommand{\e}{{\rm e}}

\newtheorem{theorem}{\bf Theorem}[section]
\newtheorem{lemma}{\bf Lemma}[section]
\newtheorem{proposition}{\bf Proposition}[section]

\theoremstyle{remark}
  \newtheorem{remark}{\bf Remark}[section]
\theoremstyle{definition}
  \newtheorem{definition}{Definition}[section]

 \numberwithin{equation}{section}
 
\begin{document}
\title[Blow-up for semilinear wave equations with non-effective damping]{
Strauss and Glassey exponents for semilinear wave equations with a  non-effective and not scattering producing damping}

\author{ Wanderley Nunes do Nascimento}
 \address{Wanderley Nunes do Nascimento, Department of Pure and Applied Mathematics, Federal University of Rio Grande do Sul, RS, 91509-900, Brazil}

\author{Alessandro Palmieri}
\address{Alessandro Palmieri, Department of Mathematics, University of Bari, 70125, Bari, Italy}

%
 \begin{abstract}
In this paper, we study the Cauchy problems for a semilinear damped wave equation with a time-dependent coefficient for the damping term belonging to the class of non-effective damping terms with critical decay rates involving iterated logarithmic factors. As nonlinearities we consider both $|u|^p$ and $|u_t|^p$.

Assuming nonnegative and compactly supported initial data, we establish the blow-up for weak solutions in the sub-Strauss range $1 < p \leq p_{\mathrm{Str}}(n)$ for the power of the nonlinear term  $|u|^p$.
The proof in the sub-critical case relies on an iteration frame for the space average of the solution, obtained by employing a time-dependent multiplier related to the coefficient of the damping term. 
On the other hand, in the limit case $p=p_{\mathrm{Str}}(n)$, we work with solutions of the homogeneous equation with separated variables, and we investigate the properties of a fundamental system of solutions for the corresponding time-dependent ODE, in order to derive an iteration frame for a suitable weighted space average of the solution.
Finally, for the derivative type nonlinearity  $|u_t|^p$ we prove the blow-up of weak solutions in the sub-Glassey range $1 < p \leq \frac{n+1}{n-1}$ by using a comparison argument for a suitable time-dependent function associated with the corresponding local in time solution.

\end{abstract}
\keywords{damped wave equation, Strauss and Glassey exponents, blow-up and lifespan estimates, time-dependent coefficients
}

\subjclass[2020]{35L15, 35L71, 35B33, 35B44}


\maketitle

\section{Introduction}

In this paper, we investigate blow-up results for the initial value problems associated with a semilinear damped wave equation with \emph{power nonlinearity} $|u|^p$
\begin{equation}\label{eq:DPE}
\begin{cases}
		u_{tt} - \Delta u + b(t) u_t = |u|^p, & t \in (0,T),\ x \in \mathbb{R}^n, \\
		u(0,x)= \varepsilon u_0(x), &x \in \mathbb{R}^n, \\  u_t(0,x)=\varepsilon u_1(x), & x \in \mathbb{R}^n,
	\end{cases}
\end{equation}
and with \emph{derivative type nonlinearity} $|u_t|^p$
\begin{equation}\label{eq:DPE der}
\begin{cases}
		u_{tt} - \Delta u + b(t) u_t = |u_t|^p, & t \in (0,T),\ x \in \mathbb{R}^n, \\
		u(0,x)= \varepsilon u_0(x), &x \in \mathbb{R}^n, \\  u_t(0,x)=\varepsilon u_1(x), & x \in \mathbb{R}^n,
	\end{cases}
\end{equation}
where $\varepsilon>0$ is a parameter describing the size of the Cauchy data and $p>1$. \\
As coefficient for the damping $b(t)u_t$ we consider 
\begin{align}\label{def b(t)}
b(t) :=\frac{\mu }{(\mathrm{e}^{[k]}+t)\prod_{j=1}^k \ln^{[j]}(\mathrm{e}^{[k]}+t)},
\end{align} where $k\in\mathbb{N}, k\geq 1$ and $\mu>0$.
Here we adopt the standard notation for iterated logarithmic and exponential functions, namely,
\begin{align*}
\ln^{[j]}(r):= \begin{cases}
\ln(r) & \mbox{if } j=1, \\
\ln(\ln^{[j-1]}(r)) & \mbox{if } j\geq 2,
\end{cases} \qquad \qquad
\mathrm{e}^{[j]}:= \begin{cases}
\mathrm{e} & \mbox{if } j=1, \\
\mathrm{e}^{\mathrm{e}^{[j-1]}} & \mbox{if } j\geq 2.
\end{cases}
\end{align*}
In order to properly frame the class of damping terms we are interested in, we need first to recall the classification for the damping terms introduced by Wirth in \cite{Wirth07,Wirth06,WirthPhD}. 

If $b\in L^1([0,+\infty))$ is a nonnegative function, for the solutions of $u_{tt}-\Delta u +b(t)u_t=0$  a scattering result in the sense of Lax-Philipps theory is proved in \cite[Section 3.1.1]{WirthPhD}. For this reason, we refer to $b(t)u_t$ as a \emph{scattering producing damping term} whenever $b\in L^1([0,+\infty))$. 

If $b(t)=\frac{\mu}{1+t}$, with $\mu>0$, the damping term $b(t)u_t$ is called \emph{scale-invariant}, as the equation 
\begin{equation}\label{scale-invariant equation}
u_{tt}-\Delta u +\tfrac{\mu}{1+t}u_t=0
\end{equation} is invariant under the so-called hyperbolic scaling. In other words, if $u$ solves \eqref{scale-invariant equation}, then, $u_\lambda$ solves \eqref{scale-invariant equation} for any $\lambda>0$  as well, where $u_\lambda(t,x):=u(\lambda(1+t)-1,\lambda x)$. It is worth noticing that the damped wave equation with scale-invariant damping is, in some sense, a non singular version of the Euler-Poisson-Darboux equation.

We remark that $\frac{\mu}{(1+t)^\kappa}u_t$ belongs to the class of scattering producing terms for $\kappa>1$, while the scale-invariant case $\kappa=1$ does not.

In \cite{Wirth06} a class of damping terms is introduced, called \emph{non-effective damping terms}, for which  Strichartz-type estimates (proved by using the stationary phase method) present similar features to those for the classical undamped wave equation. Among the scale-invariant coefficients, only those with $\mu\in (0,1)$ belong to the non-effective case. This fact is essentially due to the representation of the Fourier transform of the fundamental solutions in the phase space through $\mu$-dependent Bessel functions (cf. \cite{Wirth04}). 

In our analysis we consider the class of coefficients $$\left\{\frac{\mu} {(\mathrm{e}^{[k]}+t)\prod_{j=1}^k \ln^{[j]}(\mathrm{e}^{[k]}+t)}\right\}_{ \mu >0, \, k\in\mathbb{N}, k\geq 1}$$ whose elements lie on the edge between scale-invariant and scattering producing dampings.

 We stress that for coefficients in the above class we obtain, according to Wirth's classification, non-effective damping terms which are not scattering producing. Therefore, the significance of the choice of $b(t)$ as in \eqref{def b(t)} is due to the fact that this $b(t)$ decays faster than in the scale-invariant case but not fast enough to be in the scattering producing case. 
 
 Naively, we could expect \eqref{eq:DPE} and \eqref{eq:DPE der} to have the same critical exponents as for the corresponding classical semilinear wave equations (as we are going to recall below, this is exactly what happens when $b(t)$ is summable). In the present work, our main goal is to prove the necessity part of this conjecture: this is achieved by proving the blow-up in finite time of weak solutions to \eqref{eq:DPE} [respectively, to \eqref{eq:DPE der}] in the sub-Strauss range [respectively, in the sub-Glassey range]. 
 Furthermore, as a byproduct of the proofs of our blow-up results, we obtain upper bound estimates for the \emph{lifespan} (i.e., the maximal existence time) of local-in-time solutions expressed by means of the parameter $\varepsilon$ that quantifies the magnitude of the Cauchy data.
 
Let us now briefly review the literature on the critical exponents for the problems in \eqref{eq:DPE} and \eqref{eq:DPE der} under different assumptions on the coefficient $b(t)$. We begin with the cases with constant coefficients (classical wave and damped wave equations).

For the classical semilinear wave equation with power nonlinearity, i.e., for \eqref{eq:DPE} with $b\equiv 0$, nowadays, it is well known that the critical exponent (the threshold value for $p$ that separates the blow-up region from the global existence region of small data solutions) is given by the \emph{Strauss exponent} $p_{\mathrm{Str}}(n)$, which is the positive root of the quadratic equation $\gamma(n,p)=0$, where 
\begin{align} \label{QuadStr}
\gamma (n, p) := 1 + \tfrac{n+1}{2}p - \tfrac{n-1}{2}p^2.
\end{align}
 This exponent is named after the author of \cite{Strauss1981}, who first conjectured its value in the general $n$-dimensional case. The proof of the validity of Strauss' conjecture took about thirty years and the efforts of many authors, for the blow-up results we refer to \cite{G81,John1979,Kato80,Si,Sch85,Tak15,TW11, YorZhan06,Zhou07,ZH14} while for the global existence results we quote \cite{GLS,G2,John1979,Kubo96,LS,Zhou95}. We refer to the introduction of \cite{Tak15} (and the references therein) for a broader overview on the Strauss' conjecture and on the lifespan estimates in the subcritical and critical case. We point out that we can formally summarize the previous results in the 1-dimensional case by denoting  $p_{\mathrm{Str}}(1):=+\infty$ (this means simply that when $n=1$ a blow-up result can be shown, under suitable sign conditions for the Cauchy data, for any $p>1$).

On the other hand, for the classical semilinear wave equation with derivative type nonlinearity, i.e., for \eqref{eq:DPE der} with $b\equiv 0$, the critical exponent is the so-called \emph{Glassey exponent} $p_{\mathrm{Gla}}(n):=\frac{n+1}{n-1}$ for $n\geq 2$ (also in this case, we formally define $p_{\mathrm{Gla}}(1):=+\infty$). For the blow-up result we cite \cite{Age91,John81,Ma83,Ra87,Sid83,Sch86,Zhou01}, while for the global existence result of small data solutions we refer to \cite{HT95,HWY12,Sid83,Tzv98}.

For the classical semilinear damped wave equation (i.e., for \eqref{eq:DPE} with $b\equiv 1$) the critical exponent is completely different from the Strauss exponent: indeed, it is the celebrated \emph{Fujita exponent} $p_{\mathrm{Fuj}}(n):=1+\frac{2}{n}$ (named after the author of \cite{Fujita1966}, who found out that it was the critical exponent for the Cauchy problem associated to the semilinear heat equation with power nonlinearity). We refer to \cite{IT05,LZ95,Mat76,TY01,Zha01} and the references therein for the critical exponent of the semilinear damped wave equation.

Next, we briefly review the literature on semilinear damped wave equations with time-dependent coefficients.

 Historically, the first model to be considered was the Cauchy problem \eqref{eq:DPE} when the coefficient $b=b(t)$ provides an \emph{effective damping}, cf. \cite{Wirth07,WirthPhD}. In this case, the Fujita exponent was established as critical exponent for \eqref{eq:DPE}, see \cite{DL13, DLR13, ISW19, LinNish}. 
 
 In the scattering producing case $b\in L^1([0,+\infty))$, the situation is quite different, since the critical exponents for \eqref{eq:DPE} and \eqref{eq:DPE der} are the Strauss 
 and Glassey exponents, respectively, as for the corresponding classical semilinear wave equations (cf. \cite{LaiTakamura2019, TakTak18, WakaYard19}). 
 
 In the scale-invariant case $b(t)=\frac{\mu}{1+t}$, the critical exponents for the Cauchy problems in \eqref{eq:DPE} and \eqref{eq:DPE der} are given by $\max\{p_{\mathrm{Fuj}}(n),p_{\mathrm{Str}}(n+\mu)\}$ (cf. \cite{Dab21,Dab15,DL15,DLR15,IS18,Lai20,LPT26,LTW17,LZ21,Pal25,Pal19,PR19,PT19,TL19,Wak14}) and $p_{\mathrm{Gla}}(n+\mu)$ (cf. \cite{FH26,HH21,PT21}), respectively. In particular, for the case with power nonlinearity, the competition between the Fujita exponent and the shift of the Strauss exponent can be understood thanks to the fact that the scale-invariant case is in some sense an intermediate case between the effective and the scattering producing case. In our main results, when the coefficient of the damping term is given by \eqref{def b(t)}, we will recover these results by formally taking $\mu=0$.

Finally, we report that in the \emph{overdamping case}, i.e., for $1/b\in L^1([0,+\infty))$, the critical exponent for \eqref{eq:DPE} is 1 (meaning that a global existence result can be proved for any $p>1$ in the energy subcritical case), see  \cite{IW20,Ni19}. Moreover, in \cite{ISW19}, some cases for $b=b(t)$ on the edge between the overdamping and the effective case are studied (the Fujita exponent plays an important role in this case).

\paragraph*{Notations} We denote $B_R(0):= \{ x \in \mathbb{R}^n: |x| \leq R \}$. We write $f\lesssim g$ when there exists a positive constant $C$ such that $f\leq Cg$ and $f\approx g$ when $g\lesssim f\lesssim g$.

\section{Main results}
Before stating the main theorems, we introduce the notion of weak solution to the Cauchy problem 
\begin{equation}\label{eq:DPE general}
\begin{cases}
		u_{tt} - \Delta u + b(t) u_t = f(u,u_t), & t \in (0,T),\ x \in \mathbb{R}^n, \\
		u(0,x)= \varepsilon u_0(x), &x \in \mathbb{R}^n, \\  u_t(0,x)=\varepsilon u_1(x), & x \in \mathbb{R}^n,
	\end{cases}
\end{equation}
 that we are going to consider in our blow-up results, where the nonlinear term is either $f(u,u_t)=|u|^p$ or $f(u,u_t)=|u_t|^p$. 
\begin{definition}\label{Def weak sol}
    Let $\varepsilon > 0$ and let $(u_0, u_1) \in W^{1,1}_{\mathrm{loc}}(\mathbb{R}^n) \times L^{1}_{\mathrm{loc}}(\mathbb{R}^n)$ be compactly supported in $B_R(0)$ for some $R>0$. We say that 
    \begin{align}
        u \in \mathscr{C}([0, T),  W^{1,1}_{\mathrm{loc}}(\mathbb{R}^n) ) \cap \mathscr{C}^1([0, T),  L^{1}_{\mathrm{loc}}(\mathbb{R}^n) ) \quad \mbox{such that} \quad f(u,u_t) \in L^1_{\mathrm{loc}}([0, T),  \mathbb{R}^n )
    \end{align}
    is a \emph{weak solution} to \eqref{eq:DPE general} on $[0, T)$ if 
    \begin{enumerate}
        \item for all $t \in(0,T)$: $\operatorname{supp} u(t,\cdot) \subset B_{R+t}(0)$;
        \item $u(0, \cdot) = \varepsilon u_0$ in $L^1_{\mathrm{loc}} (\mathbb{R}^n)$;
        \item the following integral relation holds 
       \begin{align}\label{SolDef}
    & \int_0^t \int_{\mathbb{R}^n} 
    \Big( -u_t(s,x)\psi_s(s,x) 
    + \nabla u(s,x)\cdot \nabla \psi(s,x) 
    + b(s)u_t(s,x)\psi(s,x) \Big)\,\mathrm{d}x\,\mathrm{d}s \nonumber\\
    & \ +\int_{\mathbb{R}^n} u_t(t,x)\psi(t,x)\,\mathrm{d}x 
    =  \varepsilon \int_{\mathbb{R}^n} u_1(x)\psi(0,x)\,\mathrm{d}x+ \int_0^t \int_{\mathbb{R}^n} f(u(s,x),u_t(s,x)) \psi(s,x)\,\mathrm{d}x\,\mathrm{d}s
\end{align} for all $t \in (0,T)$ and all $\psi \in \mathscr{C}_0^\infty([0,T) \times \mathbb{R}^n)$.
\end{enumerate}
\end{definition}
\begin{remark}
    Employing integration by parts in \eqref{SolDef}, we get
 \begin{align}\label{WeakSol}
    &\int_{\mathbb{R}^n} 
    \Big( u_t(t,x)\psi(t,x) 
    - u(t,x)\psi_{s}(t,x) 
    + b(t)u(t,x)\psi(t,x) \Big)\,\mathrm{d}x \nonumber\\
    &\quad + \int_0^t \int_{\mathbb{R}^n} 
    u(s,x) 
    \Big( \psi_{ss}(s,x) 
    - \Delta \psi(s,x) 
    - \partial_s \big( b(s)\psi(s,x) \big) \Big)\,\mathrm{d}x\,\mathrm{d}s \nonumber\\
    &=  \varepsilon \int_{\mathbb{R}^n} 
    \Big( u_1(x)\psi(0,x) 
    - u_0(x)\psi_s(0,x) 
    + b(0)u_0(x)\psi(0,x) \Big)\,\mathrm{d}x + \int_0^t \int_{\mathbb{R}^n}  f(u(s,x),u_t(s,x)) \psi(s,x)\,\mathrm{d}x\,\mathrm{d}s 
\end{align} for all $t \in (0,T)$ and all $\psi \in \mathscr{C}_0^\infty([0,T) \times \mathbb{R}^n)$. \\ \noindent Moreover, since $u(s,\cdot)$ has compact support for each $s$, one may also allow test functions $\phi \in \mathscr{C}^\infty([0,T) \times \mathbb{R}^n)$.
\end{remark}
We point out that the class of weak solutions introduced in Definition \ref{Def weak sol} is the broadest class to which our blow-up technique  can be applied. For the sake of simplicity, we keep calling weak solutions these solutions, even though we are assuming additional regularity with respect to the time variable in comparison to the standard weak solutions.

Let us begin with the blow-up results with power nonlinearity $|u|^p$.

\begin{theorem} \label{Thm subcritical Strauss}
Let $n\geq 1$, $\mu>0$, $k\in\mathbb{N}, k\geq 1$,  $\varepsilon>0$, and  let $$\begin{cases} 1<p &  \mbox{if}  \  n=1, \\1<p<p_{\mathrm{Str}}(n) & \mbox{if} \ n\geq 2.\end{cases}$$
Let us assume that $u_0 \in W^{1,1}_{\mathrm{loc}}(\mathbb{R}^n)$ and $u_1 \in  L^{1}_{\mathrm{loc}}(\mathbb{R}^n)$ are nonnegative, compactly supported functions with $\operatorname{supp}(u_0,u_1)\subset B_R(0)$ for some $R>0$ and that $u_0$ does not vanish identically. 

\noindent Let $u$ be a weak solution to \eqref{eq:DPE} on $[0,T_\varepsilon)$ according to Definition \ref{Def weak sol}. 
Then, $u$ blows up in finite time, 
and there exists $\varepsilon_0=\varepsilon_0(n,p,\mu,k,u_0,u_1,R)>0$ such that for any $\varepsilon\in (0,\varepsilon_0]$ the lifespan $T_\varepsilon$ of $u$ satisfies 
\begin{align}\label{lifespan estimate subcritical case}
T_\varepsilon \left(\ln^{[k]}\left(\mathrm{e}^{[k]}+T_\varepsilon\right)\right)^{-\frac{\mu p(p-1)}{\gamma(n,p)}} \leq C \varepsilon^{-\frac{ p(p-1)}{\gamma(n,p)}},
\end{align} where $C$ is a positive constant independent of $\varepsilon$ and $\gamma(n,p)$ is defined in \eqref{QuadStr}.
\end{theorem}

\begin{remark} The presence of the damping term with time-dependent coefficient given by \eqref{def b(t)} increases the upper bound for the  lifespan in comparison to the one for the semilinear wave equation.
\end{remark}

In the critical regime for \eqref{eq:DPE}, namely when the exponent coincides with the Strauss exponent $p=p_{\mathrm{Str}}(n)$, one expects that solutions cannot exist globally in time, even for small initial data. The following result describes this phenomenon, showing that the corresponding solution necessarily blows up in finite time under suitable assumptions on the initial data for $k\geq 2$. In the case $k=1$, the size of the multiplicative constant $\mu$ in \eqref{def b(t)} plays a crucial role: if $\mu$ is below a certain $n$-dependent threshold, then we are able to prove the blow-up even in the critical Strauss case.
\begin{theorem}\label{Thm:CriticalCase}
Let $n \geq 2$, $\mu>0$, $k\in\mathbb{N}, k\geq 1$, $\varepsilon>0$ and consider $p=p_{\mathrm{Str}}(n)$. If $k=1$, we further assume that $\mu\in (0,\mu_0^*(n))$, where  $\mu_0^*(n):=\frac{1}{p(p-1)}$.  Let $u_0 \in W^{1,1}_{\mathrm{loc}}(\mathbb{R}^n)$ and $u_1 \in  L^{1}_{\mathrm{loc}}(\mathbb{R}^n)$ be nonnegative, not identically zero, and compactly supported functions with $\operatorname{supp}(u_0,u_1)\subset B_R(0)$ for some $R>0$. 
\\ Let $u$ be a weak solution to \eqref{eq:DPE} on $[0,T_\varepsilon)$ according to Definition \ref{Def weak sol}. 
Then, $u$ blows up in finite time and there exists a constant $C>0$, independent of $\varepsilon$, such that,
\begin{enumerate}
\item for $k\geq 2$
\begin{align}\label{Upper bound lifespan crit k>1}
\ln T_\varepsilon \left(\ln^{[k]}\left(\mathrm{e}^{[k]}+T_\varepsilon\right)\right)^{-\mu p(p-1)} \leq C \varepsilon^{- p(p-1)},
\end{align}
\item for $k=1$ and $\mu\in (0,\mu_0^*(n))$
\begin{align}\label{Upper bound lifespan crit k=1}
\ln T_\varepsilon  \leq C \varepsilon^{- \frac{p(p-1)}{1-\mu p(p-1)}}.
\end{align} 
\end{enumerate}
\end{theorem}

\begin{remark} Let us derive a different representation for the threshold value  $\mu_0^*(n)$. Since $p=p_{\mathrm{Str}}(n)$ is a solution to the quadratic equation $\frac{n-1}{2}p^2-\frac{n+1}{2}p-1=0$, we have that $p(p-1)=\frac{2}{n-1}(p+1)$. Therefore, $$\mu_0^*(n)=\frac{n-1}{2(p_{\mathrm{Str}}(n)+1)}.$$ If we plug in the explicit representation for the Strauss exponent $p_{\mathrm{Str}}(n)=\frac{n+1+\sqrt{n^2+10n-7}}{2(n-1)}$, we find
\begin{align*}
\mu_0^*(n)
=\frac{1}{8}\left(3n-1-\sqrt{n^2+10n-7}\right).
\end{align*}In particular, from the previous representation, we find that $\mu_0^*(n)=\frac{n}{4}+O(1)$ as $n\to +\infty$.
\end{remark}

Next, we consider the blow-up results with derivative type nonlinearity $|u_t|^p$.

\begin{theorem} \label{Thm subcritical Glassey}
Let $n\geq 1$, $\mu>0$, $k\in\mathbb{N}, k\geq 1$,  $\varepsilon>0$, and  let $$\begin{cases} 1<p &  \mbox{if}  \  n=1, \\1<p<p_{\mathrm{Gla}}(n) & \mbox{if} \ n\geq 2.\end{cases}$$
Let us assume that $u_0 \in W^{1,1}_{\mathrm{loc}}(\mathbb{R}^n)$ and $u_1 \in  L^{1}_{\mathrm{loc}}(\mathbb{R}^n)$ are nonnegative, nontrivial, compactly supported functions with $\operatorname{supp}(u_0,u_1)\subset B_R(0)$ for some $R>0$.

\noindent Let $u$ be a weak solution to \eqref{eq:DPE der} on $[0,T_\varepsilon)$ according to Definition \ref{Def weak sol}. 
Then, $u$ blows up in finite time, 
and there exists $\varepsilon_0=\varepsilon_0(n,p,\mu,k,u_0,u_1,R)>0$ such that for any $\varepsilon\in (0,\varepsilon_0]$ the lifespan $T_\varepsilon$ of $u$ satisfies 
\begin{align}\label{lifespan estimate subcritical case Gla}
T_\varepsilon \left(\ln^{[k]}\left(\mathrm{e}^{[k]}+T_\varepsilon\right)\right)^{-\mu\left(\frac{1}{p-1}-\frac{2}{n-1}\right)^{-1}} \leq C \varepsilon^{-\left(\frac{1}{p-1}-\frac{2}{n-1}\right)^{-1}},
\end{align} where $C$ is a positive constant independent of $\varepsilon$
\end{theorem}

\begin{theorem}\label{Thm critical Glassey}
Let $n \geq 2$, $\mu>0$, $k\in\mathbb{N}, k\geq 1$, $\varepsilon>0$ and consider $p=p_{\mathrm{Gla}}(n)$. If $k=1$, we further assume that $\mu\in (0,\mu_1^*(n)]$, where  $\mu_1^*(n):=\frac{1}{p-1}=\frac{n-1}{2}$. Let  us assume that $u_0 \in W^{1,1}_{\mathrm{loc}}(\mathbb{R}^n)$ and $u_1 \in  L^{1}_{\mathrm{loc}}(\mathbb{R}^n)$ are nonnegative, nontrivial, and compactly supported functions with $\operatorname{supp}(u_0,u_1)\subset B_R(0)$ for some $R>0$. 
\\ Let $u$ be a weak solution to \eqref{eq:DPE der} on $[0,T_\varepsilon)$ according to Definition \ref{Def weak sol}. 
Then, $u$ blows up in finite time and there exists a constant $C>0$, independent of $\varepsilon$, such that,
\begin{enumerate}
\item for $k\geq 2$
\begin{align}\label{Upper bound lifespan crit k>1 Gla}
\ln T_\varepsilon \left(\ln^{[k]}\left(\mathrm{e}^{[k]}+T_\varepsilon\right)\right)^{-\mu (p-1)} \leq C \varepsilon^{- (p-1)},
\end{align}
\item for $k=1$ and $\mu\in (0,\mu_1^*(n)]$
\begin{align}\label{Upper bound lifespan crit k=1 Gla}
\begin{cases}
\ln T_\varepsilon  \leq C \varepsilon^{- \frac{p-1}{1-\mu (p-1)}} & \mbox{if} \ \mu \in (0,\frac{n-1}{2}),\\ 
\ln \ln T_\varepsilon  \leq C \varepsilon^{- (p-1)}  & \mbox{if} \ \mu=\frac{n-1}{2}.
\end{cases}
\end{align} 
\end{enumerate}
\end{theorem}

\begin{remark} When $k=1$ the conditions $\mu\in(0,\mu_0^*(n))$ in Theorem \ref{Thm:CriticalCase} and $\mu\in(0,\mu_1^*(n)]$ in Theorem \ref{Thm critical Glassey}  are in some sense the counterpart of what happens for the scale-invariant case (which can be somehow considered as the case $k=0$): as we recalled in the introduction, in the scale-invariant case $b(t)=\frac{\mu}{1+t}$, the  constant $\mu$ has a crucial role in determining the critical exponent, while in the case $b(t)=\frac{\mu}{(\mathrm{e}+t)\ln(\mathrm{e}+t)}$ the constant $\mu$ affects just the critical cases. \\ For $k\geq 2$, we have that the role of $\mu$ cannot be captured by the scale of nonlinear terms $\{|u|^p\}_{1<p\leq p_{\mathrm{Str}}(n)}$ and $\{|u_t|^p\}_{1<p\leq p_{\mathrm{Gla}}(n)}$ (other than in the upper bound estimates for the lifespan). \\ Nevertheless, it is natural to conjecture that the multiplicative constant $\mu$ might show off its influence in the scale of nonlinear terms 
\begin{align*}
& \{|u|^{p_{\mathrm{Str}}(n)}\nu(|u|) : \nu \mbox{ modulus of continuity}\} \\
& \{|u_t|^{p_{\mathrm{Gla}}(n)}\nu(|u_t|): \nu \mbox{ modulus of continuity}\}
\end{align*} for a modulus of continuity $\nu$ such that $\tau^{-\alpha}\nu(\tau)\to +\infty$ as $\tau\to 0^+$ for any $\alpha\in (0,1]$ (i.e. for a modulus of continuity that does not belong to any H\"older class).
\end{remark}

The paper is organized as follows: in Section \ref{Section sub strauss}, we prove the subcritical case from Theorem \ref{Thm subcritical Strauss} by employing the approach with the space average of a local solution introduced in \cite{Kato80}, following the main ideas from \cite{YorZhan06}; in Section \ref{Section strauss}, we prove the critical case in Theorem \ref{Thm:CriticalCase} by using the refined approach with a suitable weighted space average of a local solution in the formulation proposed by \cite{WakaYard19,WakaYard18}; finally, in Section \ref{Section sub Glassey}, we prove Theorems \ref{Thm subcritical Glassey} and \ref{Thm critical Glassey} by working with a suitable time-dependent function associated with a local solution, following the approach in \cite{LaiTakamura2019,LZ2017}.

\section{Sub-critical case for the power nonlinearity}\label{Section sub strauss}
\subsection{The iteration frame}
Let \(u=u(t,x)\) be a weak solution to
\eqref{eq:DPE} on \([0,T_\varepsilon)\) in the sense of
Definition~\ref{Def weak sol}.
We consider the spatial average of $u$
\begin{equation}
    U(t) := \int_{\mathbb{R}^n} u(t,x)\,\mathrm{d}x.
\end{equation}
We choose a test function $\psi$ in \eqref{SolDef} such that $\psi \equiv 1$ on the truncated light-cone $\{(s,x)\in [0,t] \times \mathbb{R}^n : |x|\leq s+R\}$.
With this choice, the identity in \eqref{SolDef} reduces to
\begin{align*}
    \int_{\mathbb{R}^n} u_t(t,x)\,\mathrm{d}x 
    - \varepsilon \int_{\mathbb{R}^n} u_1(x)\,\mathrm{d}x
    + \int_0^t \int_{\mathbb{R}^n} b(s)u_t(s,x)\,\mathrm{d}x\,\mathrm{d}s
    = \int_0^t \int_{\mathbb{R}^n} |u(s,x)|^p \,\mathrm{d}x\,\mathrm{d}s,
\end{align*}
for all $t\in[0,T_\varepsilon)$. In other words, the functional $U(t)$ satisfies
\begin{equation*}
    U'(t) - U'(0) + \int_0^t b(s)U'(s)\,\mathrm{d}s
    = \int_0^t \int_{\mathbb{R}^n} |u(s,x)|^p\,\mathrm{d}x\,\mathrm{d}s.
\end{equation*}
Differentiating with respect to $t$, we arrive at the second-order differential equation
\begin{equation}\label{DifEqu}
    U''(t) + b(t)U'(t)
    = \int_{\mathbb{R}^n} |u(t,x)|^p\,\mathrm{d}x.
\end{equation}
Adapting the strategy from \cite{TakTak18,LTW17} to our coefficient \eqref{def b(t)}, we now introduce the time-dependent multiplier
\begin{equation} \label{multiplier}
    m(t) := \exp\left( \int_0^t b(s)\,\mathrm{d}s \right)
    = \exp\Big( \mu \ln \big( \ln^{[k]}(\mathrm{e}^{[k]}+t) \big) \Big)
    = \big( \ln^{[k]}(\mathrm{e}^{[k]}+t) \big)^{\mu}.
\end{equation}
Clearly, $m(t)$ is an increasing function. 
Since $m'(t)=b(t)m(t)$, multiplying \eqref{DifEqu} by $m(t)$ yields
\[
\frac{d}{dt}\big( m(t)U'(t) \big)
= m(t) \int_{\mathbb{R}^n} |u(t,x)|^p\,\mathrm{d}x.
\]
Integrating over $[0,t]$ the last equation, we obtain
\begin{equation*}
    m(t)U'(t) - U'(0)
    = \int_0^t m(s) \int_{\mathbb{R}^n} |u(s,x)|^p\,\mathrm{d}x\,\mathrm{d}s,
\end{equation*}
that is,
\begin{equation*}
    U'(t)
    = \frac{U'(0)}{m(t)}
    + \frac{1}{m(t)} \int_0^t m(s) \int_{\mathbb{R}^n} |u(s,x)|^p\,\mathrm{d}x\,\mathrm{d}s.
\end{equation*}
Integrating once more over $[0,t]$, we derive the representation formula
\begin{equation} \label{(1)}
    U(t)
    = U(0)
    + U'(0)\int_0^t \frac{\mathrm{d}s}{m(s)}
    + \int_0^t \int_0^\tau
    \frac{m(s)}{m(\tau)}
    \int_{\mathbb{R}^n} |u(s,x)|^p\,\mathrm{d}x\,\mathrm{d}s\,\mathrm{d}\tau.
\end{equation}
At this point, we use the assumptions on the initial data $u_0$ and $u_1$ and the nonnegativity of the nonlinear term. Hence, \eqref{(1)} implies that
\begin{equation} \label{PosofU}
    U(t) \geq U(0) + U'(0) \frac{t}{m(t)}  > 0 
\end{equation} for any $t \in [0,T_\varepsilon)$.
Therefore, it follows by Hölder inequality that
\begin{align} 
    (U(t))^p = |U(t)|^p = \Big|\int_{|x|\leq R+t}u(t,x)\mathrm{d}x \Big| ^p & 
   \lesssim (R+ t)^{n(p-1)} \int_{\mathbb{R}^n }|u(t,x)|^p \mathrm{d}x.  \label{EstForU}
\end{align}
Plugging \eqref{EstForU} into \eqref{(1)}, we obtain the iteration frame for $U$
\begin{equation}\label{IterFrame}
    U(t) \geq C \int_0^t \int_0^\tau
    \frac{m(s)}{m(\tau)}
    (R+s)^{-n(p-1)}(U(s))^p \, \mathrm{d}s \, \mathrm{d}\tau.
\end{equation}
Following the techniques introduced in \cite{YorZhan06}, we define
\begin{align}
\varphi(x) & := \begin{cases} \displaystyle{\int_{\mathbb{S}^{n-1}} \mathrm{e}^{x \cdot \omega} \, d\sigma_\omega} & \mbox{if } n\geq 2, \\
\vphantom{\displaystyle{\int}} \cosh x  & \mbox{if } n=1,
\end{cases} \label{Espe varphi} 
\end{align} and 
\begin{align}
   \Phi(t,x) & := \mathrm{e}^{-t} \varphi(x) \label{EspePsi}.
\end{align}
It is well known (cf \cite[Section 2]{YorZhan06}) that
\begin{align}
    & \Delta \varphi(x)=\varphi(x), \label{PropPhi} \\
    & \varphi(x) \sim c_n |x|^{-\frac{n-1}{2}} \mathrm{e}^{|x|}
    \quad \text{as } |x| \to \infty. \label{PropPhi asym}
\end{align}
Moreover, the function $\Phi$ satisfies
\begin{equation}\label{PropPsi}
    \Phi_t = -\Phi,  \qquad \Phi_{tt} - \Delta \Phi = 0.
\end{equation}
Let us use $\Phi$ as test function in \eqref{WeakSol}. Then,
\begin{align*}
    &\int_{\mathbb{R}^n} 
    \Big( u_t(t,x)\Phi(t,x) - u(t,x)\Phi_t(t,x)  
    + b(t)u(t,x)\Phi(t,x) \Big)\,\mathrm{d}x - \int_0^t \int_{\mathbb{R}^n} 
    u(s,x)  \partial_s \big( b(s)\Phi(s,x) \big) \,\mathrm{d}x\,\mathrm{d}s \nonumber\\
    &=  \varepsilon \int_{\mathbb{R}^n} 
    \Big[ u_1(x)  +
  (1  + b(0))u_0(x) \Big] \varphi(x)\,\mathrm{d}x
+\int_0^t \int_{\mathbb{R}^n} |u(s,x)|^p \Phi(s,x)\,\mathrm{d}x\,\mathrm{d}s .
\end{align*}
We point out that we can choose $\Phi\in \mathscr{C}^\infty(\mathbb{R}^n)$ as test function thanks to the support condition for the solution $u$.
Therefore,
\begin{align} \label{U_0}
    &\int_{\mathbb{R}^n} 
    \Big[ \partial_t(u(t,x)\Phi(t,x)) 
    + 2 u(t,x)\Phi(t,x) 
    + b(t)u(t,x)\Phi(t,x) \Big]\,\mathrm{d}x 
    - \int_0^t \int_{\mathbb{R}^n} u(s,x)  \partial_s \big( b(s)\Phi(s,x) \big) \,\mathrm{d}x\,\mathrm{d}s \nonumber\\
    & =  \varepsilon \int_{\mathbb{R}^n} 
    \Big[ u_1(x)  +
  (1  + b(0))u_0(x) \Big] \varphi(x)\,\mathrm{d}x+\int_0^t \int_{\mathbb{R}^n} |u(s,x)|^p \Phi(s,x)\,\mathrm{d}x\,\mathrm{d}s .
\end{align}
In order to start the iteration procedure, we need to derive a better first lower bound for $U$ than the one in \eqref{PosofU}. According to this purpose, we introduce the following auxiliary functional:
\begin{align*}
    U_0(t) := \int_{\mathbb{R}^n}u(t,x) \Phi(t,x) \mathrm{d}x = \int_{\mathbb{R}^n}u(t,x) \mathrm{e}^{-t}\varphi(x) \mathrm{d}x.
\end{align*}
Next we derive a lower bound estimate for $U_0$.
\begin{lemma} \label{Lemma U0 lb} Let us assume that $u_0 \in W^{1,1}_{\mathrm{loc}}(\mathbb{R}^n)$ and $u_1 \in  L^{1}_{\mathrm{loc}}(\mathbb{R}^n)$ are nonnegative, compactly supported functions with $\operatorname{supp}(u_0,u_1)\subset B_R(0)$ for some $R>0$ and that $u_0$ is not identically zero. Let $u$ be a weak solution to \eqref{eq:DPE general} on $[0,T_\varepsilon)$ according to Definition \ref{Def weak sol} with a nonnegative nonlinear term $f(u,u_t)$. \\ Then, there exists $K=K(k,\mu)>0$ such that 
\begin{align} \label{lb U0}
 U_0(t) \geq \frac{K\varepsilon}{m(t)} \int_{\mathbb{R}^n} u_0(x) \varphi(x) \, \mathrm{d}x
\end{align}  for any $t\in [0,T_\varepsilon)$.
\end{lemma}
\begin{proof}
By using the definition of $U_0$, rewriting \eqref{U_0} we get the relation 
\begin{align} 
    U_0'(t) + (2 + b(t)) U_0(t) = \int_0^t \int_{\mathbb{R}^n} f(u(s,x),u_t(s,x)) \Phi(s,x)\,\mathrm{d}x\,\mathrm{d}s + \int_0^t \big( b'(s) - b(s)\big)U_0(s)\mathrm{d}s   \nonumber\\ + \varepsilon \int_{\mathbb{R}^n} 
    \Big[ u_1(x) 
    +
  (1  + b(0))u_0(x) \Big] \varphi(x)\,\mathrm{d}x.\label{inter U'0}
\end{align}
Differentiating \eqref{inter U'0} with respect to $t$ and adding $2b(t)U_0(t)$ to the resulting equation, we obtain
\begin{align*}
    U''_0(t) + 2 U'_0(t) + b(t)\big(U'_0(t) + 2 U_0(t)\big)
    = b(t)U_0(t) + \int_{\mathbb{R}^n} f(u(t,x),u_t(t,x)) \Phi(t,x)\,\mathrm{d}x.
\end{align*}
Multiplying the above identity by the multiplier $m(t)$ defined in \eqref{multiplier} and using the relation $m'(t)=b(t)m(t)$, we deduce that
\begin{align*}
    \frac{\mathrm{d}}{\mathrm{d}t} \Big( m(t)\big(U'_0(t) + 2 U_0(t)\big) \Big)
    = m(t)b(t)U_0(t) + m(t)\int_{\mathbb{R}^n} f(u(t,x),u_t(t,x)) \Phi(t,x)\,\mathrm{d}x.
\end{align*}
Integrating over $[0,t]$, we have
\begin{align*}
    m(t)\big(U'_0(t) + 2 U_0(t)\big) - \big( U'_0(0) + 2U_0(0)  \big)
    & = \int_0^t m(s)b(s)U_0(s)\mathrm{d}s  \\ &\quad + \int_0^t m(s)\int_{\mathbb{R}^n} f(u(s,x),u_t(s,x)) \Phi(s,x)\,\mathrm{d}x \,\mathrm{d}s.
    \end{align*}
It follows that
\begin{align*}
\mathrm{e}^{-2t} \frac{\mathrm{d}}{\mathrm{d}t} (\mathrm{e}^{2t} U_0(t))
    & = \frac{1 }{m(t)} \big( U'_0(0) + 2U_0(0)  \big) + \int_0^t \frac{m(s)}{m(t)}b(s)U_0(s)\,\mathrm{d}s \\ & \quad + \int_0^t \frac{m(s)}{m(t)}\int_{\mathbb{R}^n} f(u(s,x),u_t(s,x)) \Phi(s,x)\,\mathrm{d}x \,\mathrm{d}s.
\end{align*}
Multiplying the last equation by $\mathrm{e}^{2t}$ and integrating once again over $[0,t]$ and usung the nonnegativity of the nonlinear term, we find
\begin{align*}
\mathrm{e}^{2t} U_0(t) - U_0(0)
&= \big( U'_0(0) + 2U_0(0)  \big) \int_0^t \frac{\mathrm{e}^{2\tau}}{m(\tau)}\,\mathrm{d}\tau + \int_0^t \int_0^\tau \mathrm{e}^{2\tau}\frac{m(s)}{m(\tau)} b(s) U_0(s)\,\mathrm{d}s\,\mathrm{d}\tau
     \\
&\quad  + \int_0^t \int_0^\tau \mathrm{e}^{2\tau}\frac{m(s)}{m(\tau)}
    \int_{\mathbb{R}^n} f(u(s,x),u_t(s,x)) \Phi(s,x)\,\mathrm{d}x\,\mathrm{d}s\,\mathrm{d}\tau \\
    & \geq  \frac{\mathrm{e}^{2t}-1}{2m(t)}(U'_0(0) + 2U_0(0)) + \int_0^t \int_0^\tau \mathrm{e}^{2\tau}\frac{m(s)}{m(\tau)} b(s) U_0(s)\,\mathrm{d}s\,\mathrm{d}\tau.
\end{align*}
Consequently
\begin{align} 
U_0(t)  & \geq \mathrm{e}^{-2t} U_0(0) +  \frac{1-\mathrm{e}^{-2t}}{2m(t)}(U'_0(0) + 2 U_0(0)) + \int_0^t \mathrm{e}^{2(\tau - t )} \int_0^\tau \frac{m(s)}{m(\tau)} b(s) U_0(s)\,\mathrm{d}s\,\mathrm{d}\tau \notag \\
 & \geq \varepsilon \,\mathrm{e}^{-2t} \int_{\mathbb{R}^n} u_0(x) \varphi(x)\,\mathrm{d}x +  \frac{\varepsilon\,(1-\mathrm{e}^{-2t})}{2m(t)}\int_{\mathbb{R}^n} (u_0(x) + u_1(x))\varphi(x)\,\mathrm{d}x \notag \\
 & \quad + \int_0^t \mathrm{e}^{2(\tau - t )} \int_0^\tau \frac{m(s)}{m(\tau)} b(s) U_0(s)\,\mathrm{d}s\,\mathrm{d}\tau. \label{EstU0}
\end{align}
 Since we assume $u_0$ nonnegative and $u_0$ not identically zero, then $U_0(0) > 0$. By continuity, $U_0(t) > 0$ in a right neighborhood of $t=0$.
By using a standard contradiction argument, since $u_1$ is nonnegative as well, from \eqref{EstU0} it follows that $U_0(t)$  is nonnegative for any $t\in [0,T_\varepsilon)$. Therefore, neglecting the last integral term on the right-hand side of \eqref{EstU0}, we obtain \eqref{lb U0}.
\end{proof}

At this point we use the following estimate for the $L^{p'}(B_{R+t}(0))$-norm of $\Phi$ (cf. \cite[Equation (2.5)]{YorZhan06})
\begin{lemma}\label{Lemma estimate Phi in Lp'}
    There exists a constant $M=M(n,p,R)> 0$ such that for any $t\geq 0$:
\begin{align*}
\int_{|x| \leq t+ R} \left[\Phi(t,x)\right]^{p'} \mathrm{d}x\leq M (R+t)^{(n-1) \big( 1- \frac{p'}{2}\big)}.
\end{align*}
\end{lemma}
By Hölder inequality and \eqref{EstU0}, for any $t\in[0,T_\varepsilon)$ it follows that
\begin{align} \label{EstForUHol}
 \int_{\mathbb{R}^n} |u(t,x)|^p \,\mathrm{d}x &\gtrsim \big( U_0(t)  \big)^p \left[ \int_{|x|\leq R+t} \left( \Phi(t,x) \right)^{p'} \mathrm{d}x\right]^{-(p-1)} \nonumber \\
    &\gtrsim  \frac{\varepsilon^p}{(m(t))^p} (R+t)^{(n-1) \left( 1 - \frac{p}{2} \right)}.
\end{align}
Now by \eqref{(1)} and \eqref{EstForUHol}, denoting by $[1-\frac{p}{2}]_{\pm}$ the positive and the negative parts of $1-\frac{p}{2}$, respectively, we get
\begin{align*}
    U(t) & \gtrsim \int_0^t \int_0^\tau
    \frac{m(s)}{m(\tau)}
     \int_{\mathbb{R}^n} |u(s,x)|^p \,\mathrm{d}x \, \mathrm{d}s \, \mathrm{d}\tau \\
    & \gtrsim \varepsilon^p\int_0^t \int_0^\tau
    \frac{m(s)}{m(\tau)}
    (m(s))^{-p} (R+s)^{(n-1)\left( 1- \frac{p}{2} \right)} \, \mathrm{d}s \, \mathrm{d}\tau \\
   & \gtrsim \frac{\varepsilon^p}{(m(t))^p} \int_0^t \int_0^\tau
     (R+s)^{(n-1)\left( 1- \frac{p}{2} \right)} \, \mathrm{d}s \, \mathrm{d}\tau \\ 
     & \gtrsim \frac{\varepsilon^p}{(m(t))^p} (R+t)^{-(n-1)\left[ 1- \frac{p}{2} \right]_-}\int_0^t \int_0^\tau
     (R+s)^{(n-1)\left[ 1- \frac{p}{2}\right]_+} \, \mathrm{d}s \, \mathrm{d}\tau \\ 
      & \gtrsim \frac{\varepsilon^p}{(m(t))^p} (R+t)^{-(n-1)\left[ 1- \frac{p}{2} \right]_-}\int_0^t \int_0^\tau
     s^{(n-1)\left[ 1- \frac{p}{2}\right]_+} \, \mathrm{d}s \, \mathrm{d}\tau. 
\end{align*}
Then, 
\begin{align}\label{FirstLowB}
    U(t) \geq K \frac{\varepsilon^p}{(m(t))^p} (R+t)^{-(n-1)\left[ 1- \frac{p}{2}\right]_-}t^{2+(n-1)\left[1- \frac{p}{2}\right]_+},
\end{align}
for a suitable $K= K(n, p, u_0, u_1)>0$ and for any $t\in[0,T_\varepsilon)$.
\begin{remark}
We note that for $n\geq 2$
\[
(n-1)\left(1-\frac{p_{\mathrm{Str}}(n)}{2}\right)>-1
\quad \Longleftrightarrow \quad
p_{\mathrm{Str}}(n)<\frac{2n}{n-1}.
\]
We recall that for $p>1$ we have $p<p_{\mathrm{Str}}(n)$ if and only if $\frac{n-1}{2}p^2-\frac{n+1}{2}p-1<0$. Since
\[
(n-1)\left(\frac{2n}{n-1}\right)^2
-(n+1)\frac{2n}{n-1}-2
=2(n-1)>0
\]
for $n\geq 2$, this guaranteed that $(n-1)\left(1-\tfrac{p_{\mathrm{Str}}(n)}{2}\right)>-1$. Notice that, the previous inequality is obviously true even for $n=1$.
\end{remark}

\subsection{Iteration procedure}
The blow-up result is proved by means of an iteration argument, adapting the approach in \cite{TakTak18} to the presence of the multiplier $m(t)$.
The goal is to prove the following sequence of lower bound estimates
\begin{align}\label{lbej}
    U(t) \geq C_j \varepsilon^{p^{j+1}}\left( m(t) \right)^{-p^{j+1}}t^{a_j}(R+ t)^{-b_j}
\end{align} for any $j \in \mathbb{N}$ and any $t \in [0,T_\varepsilon)$, where $\{ a_j\}_{j\in\mathbb{N}}$, $\{ b_j\}_{j\in\mathbb{N}}$ and $\{ C_j\}_{j\in\mathbb{N}}$ are sequences of nonnegative real numbers to be determined. We proceed by induction on $j$.
For $j=0$, due to \eqref{FirstLowB}, we take 
\begin{align} \label{a_0b_0c_0}
    a_0 := 2+ (n-1)\left[ 1- \frac{p}{2} \right]_+, \,\,\, b_0:=(n-1) \left[ 1- \frac{p}{2}\right]_-,  \,\,\,  C_0 := K .
    \end{align}
Now, we assume \eqref{lbej} satisfied for $j$ and we prove it for $j+1$.
Plugging \eqref{lbej} into \eqref{IterFrame}, and using that $m$ is an increasing function, we obtain
\begin{align*}
    U(t) & \geq C \int_0^t \int_0^\tau
    \frac{m(s)}{m(\tau)}
    (R+s)^{-n(p-1)} C_j^p \varepsilon^{p^{j+2}} (m(s))^{-p^{j+2}}s^{pa_j}(R+ s)^{-pb_j}  \, \mathrm{d}s \, \mathrm{d}\tau \\
    & \geq CC_j^{p} \varepsilon^{p^{j+2}}(m(t))^{-p^{j+2}} (R+ t)^{-n(p-1)-pb_j}  \int_0^t \int_0^\tau s^{pa_j} \, \mathrm{d}s \, \mathrm{d}\tau \\
    &=\frac{C C_j^{p} }{(pa_j+1)(pa_j + 2)}\varepsilon^{p^{j+2}}(m(t))^{-p^{j+2}} (R+ t)^{-n(p-1)-pb_j}  t^{pa_j+2} 
\end{align*} for any $t\in[0,T_\varepsilon).$
    Hence, we can define the sequences $\{ a_j\}_{j\in\mathbb{N}}$, $\{ b_j\}_{j\in\mathbb{N}}$ and $\{ C_j\}_{j\in\mathbb{N}}$ through the inductive relations
\begin{align} \label{Sequences}
 a_j := pa_{j-1} + 2, \,\,\, b_j := pb_{j-1} + n(p-1), \,\,\,  C_j := \frac{C C_{j-1}^{p} }{(pa_{j-1}+1)(pa_{j-1} + 2)}\equiv \frac{C C_{j-1}^{p} }{(a_{j}-1)a_{j}}.
\end{align}
It follows from \eqref{a_0b_0c_0} and \eqref{Sequences} that
\begin{align*}
     a_j&= p(p a_{j-2} +2)+2 = p^2 a_{j-2} +2(p+1) 
     = p^3 a_{j-3} + 2(p^2 + p + 1) \\
     &= \ldots = p^ja_0 + 2 (p^{j-1} + \ldots +p +1)= a_0 p^j + 2\, \tfrac{p^j-1}{p-1}  = \left(  a_0 +  \tfrac{2}{p-1}\right)p^j - \tfrac{2}{p-1}.
\end{align*}
In a similar way, 
\begin{align*} 
 b_j = (b_0 + n)p^j - n.
\end{align*}  
From the representation for $a_j$, we have
\begin{align*}
 C_j \geq \frac{CC_{j-1}^p}{a_j^2} \geq D p^{-2j}C_{j-1}^p, 
\end{align*}
where $D := C/( a_0 + \frac{2}{p-1} )^2$.
Taking the logarithm of both sides of the previous inequality, we obtain
\begin{align*}
       \ln C_j &\geq p \ln C_{j-1} -2j \ln p + \ln D \\
               & \geq p (p\ln C_{j-2} -2(j-1) \ln p + \ln D  ) -2j \ln p + \ln D \\
               &= p^2 \ln C_{j-2} - 2((j-1)p + j)  \ln p + (1+p) \ln D \\
                & \geq \ldots \geq  p^j \ln C_{0} - 2  \ln p \sum_{k=0}^{j-1}(j-k)p^k + \ln D  \sum_{k=0}^{j-1}p^k  \\
                &=  p^j \ln C_{0} - \frac{2  \ln p }{p-1}\left( \frac{p^{j+1}-p}{p-1} - j\right) + \ln D \,\frac{p^{j}-1}{p-1} \\
                 &=  p^j \left( \ln C_{0} - \frac{2p  \ln p}{(p-1)^2} +  \frac{\ln D}{p-1}\right)  + \frac{2 p \ln p}{(p-1)^2}  + \frac{2\ln p}{p-1}j - \frac{\ln D}{p-1} 
   \end{align*}
In the previous estimates, we used the identities $\sum_{k=0}^{j-1}(j-k)p^k=\frac{1}{p-1}\left(\frac{p^{j+1}-p}{p-1}-j\right)$ and $\sum_{k=0}^{j-1}p^k=\frac{p^j-1}{p-1}$. \\
We now note that if  
$j \geq  j_0 := \left\lfloor\frac{\ln D}{2\ln p}-\frac{p}{p-1}\right\rfloor_+$, then
$$\frac{2 p}{(p-1)^2} \ln p + \frac{2\ln p}{p-1}j - \frac{\ln D}{p-1} \geq 0.$$
Therefore, for any $j\geq j_0$ we have
\[  C_j \geq \exp(p^j \ln E), \]
where $E := C_0 p^{-2p/(p-1)^2}D^{1/(p-1)}$.

\noindent Using the last inequality in \eqref{lbej}, we can derive the following estimate for any $t \geq R$
\begin{align} \label{LastEstFoU}
    U(t) &\geq  \exp \left( p^j \big((\ln (E \varepsilon^p) - p \ln (m(t))\big)\right) t^{\left(  a_0 +  \frac{2}{p-1}\right)p^j - \frac{2}{p-1}}(R+ t)^{n-(b_0 + n)p^j} \nonumber\\
    &= \exp \left( p^j \left(\ln (E \varepsilon^p) - p \ln (m(t)) +\left( a_0 + \tfrac{2}{p-1}\right) \ln t -(b_0+n)\ln(R+ t) \right)\right) t^{ - \frac{2}{p-1}}(R+ t)^{n}\nonumber\\
    &\geq \exp \left( p^j \left(\ln \left(2^{-(b_0+n)}E \varepsilon^p\right) - p \ln (m(t)) +\left( a_0 - b_0+ \tfrac{2}{p-1} - n\right) \ln t \right)\right) t^{ - \frac{2}{p-1}}(R+ t)^{n}\nonumber \\
    &= \exp \left( p^j J(t,\varepsilon)\right) t^{ - \frac{2}{p-1}}(R+ t)^{n},
\end{align}
where 
\begin{align} \label{defJ}
     J(t,\varepsilon) := \ln \big( \widetilde{E}\varepsilon^p\big) - p \ln (m(t)) +\left[ a_0 - b_0+ \tfrac{2}{p-1} - n\right] \ln t,
\end{align} with  $\widetilde{E} := 2^{-(b_0+n)}E$.

\noindent We notice that, 
    \[ a_0 -b_0 =  2+ \left[(n-1)\left( 1- \frac{p}{2} \right)\right]_+ - \left[(n-1)\left( 1- \frac{p}{2}\right)\right]_- = 2+(n-1)\left( 1- \frac{p}{2}\right)    \]
and
    \[2+(n-1) \left( 1-\frac{p}{2}\right)+\frac{2}{p-1} - n = \frac{1}{p-1} \left[   - \frac{n-1}{2}p^2  + \frac{n+1}{2}p +1 \right]>0 \quad  \Longleftrightarrow \quad 1 < p < p_{\mathrm{Str}}(n).\]
Recalling the definition for $\gamma (n, p)$ in \eqref{QuadStr}, we can then rewrite $J(t,\varepsilon)$  as
\[ J(t,\varepsilon) = \ln \left(  \widetilde{E}\, \varepsilon^p (m(t))^{-p} t^{\frac{\gamma(n,p)}{p-1}}  \right).   \]
From \eqref{LastEstFoU}, we observe that the condition $J(t,\varepsilon)>1$ is sufficient to ensure the blow-up of $U(t)$ as $j\to +\infty$. This occurs whenever
\begin{align} \label{lastestim}
    (m(t))^{-p} \, t^{\frac{\gamma(n,p)}{p-1}} \geq \widetilde{E}^{-1} \varepsilon^{-p}
    \quad \Longleftrightarrow \quad
    t \, m(t)^{-\frac{p(p-1)}{\gamma(n,p)}} \geq \hat{E} \, \varepsilon^{-\frac{p(p-1)}{\gamma(n,p)}},
\end{align} where $\hat{E} := \widetilde{E}^{-\frac{(p-1)}{\gamma(n,p)}}$.
\begin{lemma} \label{Prop1}
Let $1<p< p_{\mathrm{Str}}(n)$. Let us introduce the function
\[
f(t) := t \, (m(t))^{-\frac{p(p-1)}{\gamma(n,p)}},
\]
where $\gamma(n,p)$ and $m$ are defined in \eqref{QuadStr} and in \eqref{multiplier}, respectively.

\noindent Then there exists $t_0>0$ such that $f$ is strictly increasing in $[t_0,+\infty)$. Furthermore, $\displaystyle{\lim_{t\to +\infty}f(t)=+\infty}$
\end{lemma}
\begin{proof}
Let us denote
\[
\alpha:=\frac{\mu p(p-1)}{\gamma(n,p)}.
\]
Due to the condition $1< p< p_{\mathrm{Str}}(n)$, we have that $\gamma(n,p)> 0$ and, consequently, $\alpha >0$.
We may rewrite $f$ as 
\[
f(t)=t\bigl(\ln^{[k]}(\mathrm{e}^{[k]}+t)\bigr)^{-\alpha}.
\]
Clearly, $\displaystyle{\lim_{t\to +\infty}f(t)=+\infty}$. Moreover, a direct computation yields
\[
f'(t)=\bigl(\ln^{[k]}(\mathrm{e}^{[k]}+t)\bigr)^{-\alpha}
\left(
1-\alpha\frac{t}{(\mathrm{e}^{[k]}+t)\prod_{j=1}^k \ln^{[j]}(\mathrm{e}^{[k]}+t)}
\right).
\]
Hence $f'(t)>0$ whenever
\[
\prod_{j=1}^k \ln^{[j]}(\mathrm{e}^{[k]}+t)
>
\alpha \frac{t}{\mathrm{e}^{[k]}+t}.
\]
Since
\[
\prod_{j=1}^k \ln^{[j]}(\mathrm{e}^{[k]}+t)\to+\infty \quad \mbox{and} \quad \alpha \frac{t}{\mathrm{e}^{[k]}+t} \to \alpha
\] as $t\to+\infty$, it follows that $f$ is increasing when $t$ is sufficiently large.
\end{proof}

Let us fix $\varepsilon_0=\varepsilon_0(n,p,\mu,k,u_0,u_1,R)>0$ such that $ \hat{E} \, \varepsilon_0^{-\frac{p(p-1)}{\gamma(n,p)}}\geq f(t_0)$ and $f^{-1}\Big(\hat{E} \, \varepsilon_0^{-\frac{p(p-1)}{\gamma(n,p)}}\Big)\geq R$. We underline that we can find a sufficiently small $\varepsilon_0>0$ satisfying the previous relations since $\gamma(n,p)>0$
and the fact that $f:[t_0,+\infty)\to [f(t_0),+\infty)$ is increasing and positively divergent as $t\to+\infty$ implies that  $\displaystyle{\lim_{\sigma\to +\infty}f^{-1}(\sigma)=+\infty}$.
Then, for $\varepsilon\in(0,\varepsilon_0]$ and
$ t > f^{-1} \Big(\hat{E} \, \varepsilon^{-\frac{p(p-1)}{\gamma(n,p)}}\Big) $ we have $t\geq \max\{R,t_0\}$ and $J(t,\varepsilon)>1$, so, letting $j \rightarrow \infty$ in \eqref{LastEstFoU} we obtain that the lower bound for $U(t)$ blows up. Then, we prove that $U(t)$ is not finite for such a $t$. Hence, $u$ blows up in finite time and $$T_\varepsilon\leq f^{-1} \Big(\hat{E} \, \varepsilon^{-\frac{p(p-1)}{\gamma(n,p)}}\Big),$$ that is, we established \eqref{lifespan estimate subcritical case} too.

\section{Critical case for the power nonlinearity}\label{Section strauss}

\subsection{Auxiliary functions}

Let $\varphi$ be the function defined in \eqref{Espe varphi}. We denote $\varphi_\lambda(x):= \varphi(\lambda x)$. Clearly, $\varphi_\lambda$ satisfies $\Delta \varphi_\lambda = \lambda^2 \varphi_\lambda$.

The next lemma is the counterpart of \cite[Lemma 2.3]{WakaYard19} without the summability assumption for the coefficient of the damping term.
\begin{lemma} \label{Lemma53}
Let $\lambda > 0$ and consider the second-order differential operators
\[
L_b := \partial_t^2 + b(t)\partial_t - \lambda^2,
\qquad
L_b^* := \partial_s^2 - \partial_s \left( b(s) \cdot  \right) - \lambda^2.
\]
Let $\{y_0(t,s;\lambda), y_1(t,s;\lambda)\}$ be the fundamental system of solutions to $L_b y = 0$ uniquely determined by the initial conditions
\begin{align*}
& y_0(s,s;\lambda) = 1, && \partial_t y_0(s,s;\lambda) = 0, \\
& y_1(s,s;\lambda) = 0, && \partial_t y_1(s,s;\lambda) = 1.
\end{align*}
Then $y_0,y_1$ depend continuously on $\lambda$. Moreover, for any $t \geq s \geq 0$, the following lower bounds hold:
\begin{align*}
\mathrm{(i)}  \ \
& y_0(t,s;\lambda) 
\geq 
\frac{m(s)}{m(t)}
\cosh\big(\lambda (t-s)\big), \\
\mathrm{(ii)}  \ \
& y_1(t,s;\lambda) 
\geq 
\frac{m(s)}{m(t)} \,
\frac{\sinh\big(\lambda (t-s)\big)}{\lambda}.
\end{align*}
In addition, $y_1$ satisfies the adjoint equation
\[
\mathrm{(iii)} \ \  L^* y_1(t,s;\lambda) = 0,
\]
and the following relation holds
\begin{align*}
& \mathrm{(iv)} \ \  
y_0(t,s;\lambda) 
= b(s)\,y_1(t,s;\lambda) - \partial_s y_1(t,s;\lambda).
\end{align*}
\end{lemma}

\begin{proof}
By the classical theory for ODE, there exists a fundamental system of $\mathscr{C}^\infty$-solutions $\{y_1(t,s;\lambda), y_2(t,s;\lambda)\}$, depending continuously on the parameter $\lambda$, such that
\begin{align*}
\begin{cases}
L_b y_0(t,s;\lambda) = 0, \\
y_0(s,s;\lambda) = 1, \\
\partial_t y_0(s,s;\lambda) = 0,
\end{cases} \qquad
\begin{cases}
L_b y_1(t,s;\lambda) = 0, \\
y_1(s,s;\lambda) = 0, \\
\partial_t y_1(s,s;\lambda) = 1,
\end{cases}
\end{align*}
for all $t \geq s \geq 0$.
The proof of this lemma relies on some key identities.
Let us denote $$B(t):=\int_0^t b(\tau) \, \mathrm{d}\tau = \mu \ln\left(\ln^{[k]}(\mathrm{e}^{[k]}+t)\right).$$ If $y(t;\lambda)$ is a solution to $L_by=0$, then we get
\begin{equation}\label{eq:id1}
\left( y'(t;\lambda)\mathrm{e}^{B(t)} \right)' = \lambda^2 y(t;\lambda)\mathrm{e}^{B(t)},
\end{equation} here the prime denotes the derivative with respect to $t$.
Using this relation, one obtains
\begin{equation}\label{eq:id2}
\left( y(t;\lambda)\mathrm{e}^{B(t)} - \int_s^t b(r)y(r;\lambda)\mathrm{e}^{B(r)}\,\mathrm{d}r \right)'' 
= \lambda^2 y(t;\lambda)\mathrm{e}^{B(t)}.
\end{equation}

As a first application of the key identities in \eqref{eq:id1} and \eqref{eq:id2}, we show that $y_0(t,s,\lambda)$ does not vanish for $t>s$. Indeed, suppose by contradiction that there exists $t_0>s$ such that $y_0(t_0,s;\lambda)=0$, and let $t_0$ be the smallest of these values. Then $y_0(t,s;\lambda)> 0$ for all $t\in[s,t_0]$, and from \eqref{eq:id1} we deduce
\[
y_0'(t,s;\lambda)\mathrm{e}^{B(t)} 
= \lambda^2 \int_s^t y_0(r,s;\lambda)\mathrm{e}^{B(r)}\,\mathrm{d}r \geq 0,
\]
which implies $y_0'(t,s;\lambda)\geq 0$ on $[s,t_0]$. Hence, $y_0(\cdot,s;\lambda)$ is non-decreasing on $[s,t_0]$ but this is not possible since $1=y_0(s,s;\lambda) \geq y_0(t_0,s;\lambda)=0$. This contradicts the assumption $y_0(t_0,s;\lambda)=0$. The positivity of $y_0(t,s;\lambda)$, combined with \eqref{eq:id1}, implies that 
\[
 y_0'(t,s;\lambda) \geq 0 \quad \text{for any } t \geq s.
\]
Next, we derive a lower bound for $y_0(t,s;\lambda)$. We use \eqref{eq:id2} and the fact that $y_0(t,s;\lambda)$ is positive. Setting
\[
z_0(t,s;\lambda):=y_0(t,s;\lambda)\mathrm{e}^{B(t)}-\int_s^t b(r)y_0(r,s;\lambda)\mathrm{e}^{B(r)}\,\mathrm{d}r,
\]
we infer from \eqref{eq:id2} and from the positivity of $y_0(t,s;\lambda)$ that
\begin{align*}
\partial_t^2 z_0(t,s;\lambda) & = \lambda^2 y_0(t,s;\lambda) \mathrm{e}^{B(t)} \\
& \geq \lambda^2 \left( y_0(t,s;\lambda) \mathrm{e}^{B(t)} -\int_s^t b(r) y_0(r,s;\lambda) \mathrm{e}^{B(r)}\mathrm{d}r \right) = \lambda^2 z_0(t,s;\lambda).
\end{align*}
Moreover, since $\partial_t z_0(t,s;\lambda)=\partial_t y_0(t,s;\lambda)\mathrm{e}^{B(s)}$, evaluating the initial values of $z_0(t,s;\lambda)$ at $t=s$, we find
\[
z_0(s,s;\lambda)=\mathrm{e}^{B(s)},
\qquad
\partial_t z_0(s,s;\lambda)=0.
\]
Being $z_0(t,s;\lambda)$ a super-solution of $y''=\lambda^2 y$ with initial conditions $(\mathrm{e}^{B(s)},0)$ at $t=s$, by a comparison argument we have that
\[
z_0(t,s;\lambda)\geq \mathrm{e}^{B(s)}\cosh\bigl(\lambda(t-s)\bigr).
\]
Therefore, for any $t\geq s$
\[
y_0(t,s;\lambda)\mathrm{e}^{B(t)}
-\int_s^t b(r)y_0(r,s;\lambda)\mathrm{e}^{B(r)}\,\mathrm{d}r
\geq \mathrm{e}^{B(s)}\cosh\bigl(\lambda(t-s)\bigr).
\]
Since the integral term is nonnegative and $\mathrm{e}^{B(t)}=m(t)$, it follows that for any $t\geq s$
\begin{equation}\label{eq:y1-lower}
y_0(t,s;\lambda)\geq \mathrm{e}^{B(s)-B(t)}\cosh\bigl(\lambda(t-s)\bigr) = \frac{m(s)}{m(t)}\cosh\bigl(\lambda(t-s)\bigr).
\end{equation}

We now turn to the proof of claim (ii).
Let us begin by remarking that $y_1(t,s;\lambda)$ is positive for any $t>s$. Indeed, since $\partial_t y_1(s,s;\lambda)=1$, we can find a right neighborhood $[s,s+\delta]$ of $s$ such that $\forall t\in [s,s+\delta]$: $\partial_t y_1(t,s;\lambda)>0$. Hence, since $y_1(s,s;\lambda)=0$ and $y_1(\cdot,s;\lambda)$ is increasing, it follows that $\forall t\in (s,s+\delta]$: $y_1(t,s;\lambda)>0$. Proceeding with a similar argument to the one used for $y_0(t,s;\lambda)$, we conclude that  $\forall t>s$: $y_1(t,s;\lambda)>0$. In order to derive the lower bound for $y_1$, we employ once again a comparison argument. We denote
\[
z_1(t,s;\lambda):=y_1(t,s;\lambda)\mathrm{e}^{B(t)}-\int_s^t b(r)y_1(r,s;\lambda)\mathrm{e}^{B(r)}\,\mathrm{d}r.
\]
Thanks to \eqref{eq:id2} and the nonnegativity of  $y_1(t,s;\lambda)$, we have that 
\begin{align*}
\partial_t^2 z_1(t,s;\lambda) & = \lambda^2 y_1(t,s;\lambda) \mathrm{e}^{B(t)} \\
& \geq \lambda^2 \left( y_1(t,s;\lambda) \mathrm{e}^{B(t)} -\int_s^t b(r) y_1(r,s;\lambda) \mathrm{e}^{B(r)}\mathrm{d}r \right) = \lambda^2 z_1(t,s;\lambda).
\end{align*} Furthermore, being $\partial_t z_1(t,s;\lambda)=\partial_t y_1(t,s;\lambda)\mathrm{e}^{B(s)}$, we have
\[
z_1(s,s;\lambda)=0,
\qquad
\partial_t z_1(s,s;\lambda)=\mathrm{e}^{B(s)}.
\]
By comparison with the solution to $y''=\lambda^2 y$ having the same initial values at $t=s$, we have
\begin{align*}
z_1(t,s;\lambda)\geq \mathrm{e}^{B(s)}\frac{\sinh(\lambda(t-s)}{\lambda}.
\end{align*} By the definition of $z_1(t,s;\lambda)$, using the previous inequality and being $y_1(t,s;\lambda)$ nonnegative, we find that
\begin{align*}
y_1(t,s;\lambda) & = z_1(t,s;\lambda)\mathrm{e}^{-B(t)}+\mathrm{e}^{-B(t)}\int_s^t b(r)y_1(r,s;\lambda)\mathrm{e}^{B(r)}\,\mathrm{d}r \\
 & \geq z_1(t,s;\lambda)\mathrm{e}^{-B(t)} \\
 & \geq \mathrm{e}^{B(s)-B(t)}\frac{\sinh(\lambda(t-s)}{\lambda} \\
 & = \frac{m(s)}{m(t)}\, \frac{\sinh(\lambda(t-s)}{\lambda}.
\end{align*}

In order to prove (iii) and (iv), we need to write $y_1(t,s;\lambda)$ as linear combination of $y_0(t,0;\lambda), y_1(t,0;\lambda)$. For the sake of brevity, we denote $y_j(t;\lambda)\equiv y_j(t,0;\lambda)$ for $j=0,1$. Since $\left\{y_0(t;\lambda),y_1(t;\lambda)\right\}$ is a fundamental system of solutions to $L_b y=0$, there exists a uniquely determined matrix $[c_{ij}(s;\lambda)]_{ 0\leq i,j\leq 1}$ such that $$y_j(t,s;\lambda)= \sum_{k=0}^1 c_{jk}(s;\lambda) y_k(t;\lambda)$$ for $j=0,1$. Clearly, $\partial_t y_j(t,s;\lambda)= \sum_{k=0}^1 c_{jk}(s;\lambda) y'_k(t;\lambda)$ for $j=0,1$. By using the initial conditions $\partial_t^i y_j(s,s;\lambda)=\delta_{ij}$ (here, $\delta_{ij}$ denotes Kroneker delta), we obtain
\begin{align*}
\left[ \begin{array}{cc}
c_{00}(s;\lambda) & c_{01}(s;\lambda) \\ 
c_{10}(s;\lambda) & c_{11}(s;\lambda)
\end{array} \right]  \left[ \begin{array}{cc}
y_{0}(s;\lambda) & y'_{0}(s;\lambda) \\ 
y_{1}(s;\lambda) & y'_{1}(s;\lambda)
\end{array} \right] =  \left[ \begin{array}{cc}
1 & 0 \\ 
0 & 1
\end{array} \right].
\end{align*} Since the Wronskian $W(y_0,y_1)(t;\lambda)= y_0(t;\lambda)y'_1(t;\lambda)-y_1(t;\lambda)y'_0(t;\lambda)$ solves the ODE $W'+b(t)W=0$ with $W(y_0,y_1)(0
;\lambda)=1$, we have $W(y_0,y_1)(t;\lambda)=\mathrm{e}^{-B(t)}$. Then,
\begin{align*}
\left[ \begin{array}{cc}
c_{00}(s;\lambda) & c_{01}(s;\lambda) \\ 
c_{10}(s;\lambda) & c_{11}(s;\lambda)
\end{array} \right]   =\left[ \begin{array}{cc}
y_{0}(s;\lambda) & y'_{0}(s;\lambda) \\ 
y_{1}(s;\lambda) & y'_{1}(s;\lambda)
\end{array} \right]^{-1} =\mathrm{e}^{B(s)}\left[ \begin{array}{cc}
 y'_{1}(s;\lambda) & -y'_{0}(s;\lambda) \\ 
-y_{1}(s;\lambda) & y_{0}(s;\lambda)
\end{array} \right].
\end{align*}
Therefore,
\begin{align}
y_0(t,s;\lambda) & = \mathrm{e}^{B(s)}\big(y'_1(s,\lambda)y_0(t,\lambda)-y'_0(s,\lambda)y_1(t,\lambda)\big), \label{repr y0} \\
y_1(t,s;\lambda) & = \mathrm{e}^{B(s)}\big(y_0(s,\lambda)y_1(t,\lambda)-y_1(s,\lambda)y_0(t,\lambda)\big). \label{repr y1} 
\end{align} Let us prove (iv). Differentiating \eqref{repr y1} with respect to $s$ and using \eqref{repr y0} and \eqref{repr y1}, we find
\begin{align}
\partial_s y_1(t,s;\lambda) & = b(s) \mathrm{e}^{B(s)}\big(y_0(s,\lambda)y_1(t,\lambda)-y_1(s,\lambda)y_0(t,\lambda)\big)+\mathrm{e}^{B(s)}\big(y'_0(s,\lambda)y_1(t,\lambda)-y'_1(s,\lambda)y_0(t,\lambda)\big) \notag \\ & = b(s)y_1(t,s;\lambda)-y_0(t,s;\lambda). \label{equation (iv)}
\end{align} We prove now (iii). Differentiating \eqref{equation (iv)}  with respect to $s$ and using \eqref{repr y0} and \eqref{repr y1}, we obtain
\begin{align*}
\partial^2_s y_1(t,s;\lambda) & = \partial_s\big(b(s)y_1(t,s;\lambda)-y_0(t,s;\lambda)\big) \\
&= \partial_s\big(b(s)y_1(t,s;\lambda)\big) - \mathrm{e}^{B(s)} \left[ b(s) y_0(t,s;\lambda)+\big(y''_1(s,\lambda)y_0(t,\lambda)-y''_0(s,\lambda)y_1(t,\lambda)\big)\right] \\
&= \partial_s\big(b(s)y_1(t,s;\lambda)\big) - \mathrm{e}^{B(s)} \left[ (y''_1(s;\lambda)+b(s)y'_1(s;\lambda))y_0(t;\lambda)-(y''_0(s;\lambda)+b(s)y'_0(s;\lambda))y_1(t;\lambda)\right] \\
&= \partial_s\big(b(s)y_1(t,s;\lambda)\big) - \lambda^2 \mathrm{e}^{B(s)} \left[ y_1(s;\lambda)y_0(t;\lambda)-y_0(s;\lambda)y_1(t;\lambda)\right] \\ &=\partial_s\big(b(s)y_1(t,s;\lambda)\big) + \lambda^2 y_1(t,s;\lambda),
\end{align*} that is, $L_b^* y_1=0$.
\end{proof}

Let $\lambda_0 >0$ and $q>-1$ be fixed parameters (we are going to set the exact value of $q$ afterwards). We define the auxiliary functions
\begin{align}
\xi_q(t,s, x) 
&:= \int_0^{\lambda_0} \mathrm{e}^{-\lambda(t+R)} y_0(t,s;\lambda)\, \varphi_\lambda(x)\lambda^q\, \mathrm{d}\lambda, 
\label{eq:xi_def} \\  
\eta_q(t,s,x) 
&:= \int_0^{\lambda_0} \mathrm{e}^{-\lambda(t+R)} \frac{y_1(t,s;\lambda)}{(t-s)}\, \varphi_\lambda(x)\lambda^q\, \mathrm{d}\lambda, 
\label{eq:eta_def}
\end{align}
for $0\leq s\leq t$ and $|x|\leq R+s$, where the functions $y_0,y_1$ are defined in the statement of Lemma \ref{Lemma53}. The main properties of these functions are summarized in the next lemma.

\begin{lemma} \label{Estiforauxiliarfunctions}
Let $n \geq 2$ and $\lambda_0 >0$. Then,  the following estimates hold:
\begin{enumerate}
\item[(i)] If $q>-1$, $|x|\leq R$ and $t\geq 0$, then
\begin{align*}
& \xi_q(t,0,x) \geq A_0 (m(t))^{-1} ,\\
& \eta_q(t,0, x) \geq B_0 (m(t))^{-1} \langle t\rangle^{-1} .
\end{align*}

\item[(ii)] If $q>-1$, $|x|\leq s+R$ and $0 \leq s \leq  t$, then
\begin{align*}
& \xi_q(t,s, x) \geq A_1 \frac{m(s)}{m(t)} \langle s\rangle^{-1-q}, \\
& \eta_q(t,s, x) \geq B_1 \frac{m(s)}{m(t)} \langle t\rangle^{-1} \langle s\rangle^{-q}.
\end{align*}

\item[(iii)] If $q > \frac{n-3}{2}$, $|x|\leq t+R$ and $t>0$, then
\begin{align*}
\xi_q(t,t, x) \leq B_2 \langle t\rangle^{-(n-1)/2}\langle t-|x|\rangle^{(n-3)/2 - q}.
\end{align*}
\end{enumerate}

Here $A_0,A_1$ and $B_0,B_1,B_2$ are positive constants depending only on $\alpha$, $q$, $R$, and $\langle y\rangle := 3+|y|$ $\forall y\in\mathbb{R}$.
\end{lemma}
The proof can be done by miming the one in \cite[Lemma 3.1]{WakaYard18}, employing the lower bound estimates for $y_0,y_1$ from Lemma \ref{Lemma53}.

\subsection{Nonlinear integral inequality}
In this subsection, we introduce the time dependent functional, whose evolution in time will provide the blow-up result, and we derive a nonlinear integral inequality involving the functional, that will be used as iteration frame. Moreover, we establish a first lower bound estimate for the functional to be used as starting estimate in the iteration argument in the next subsection.

\begin{proposition} \label{Prop51}
Let $u_0 \in W^{1,1}_{\mathrm{loc}}(\mathbb{R}^n)$, $u_1 \in  L^{1}_{\mathrm{loc}}(\mathbb{R}^n)$ be nonnegative, not identically zero, and compactly supported functions such that $\mathrm{supp}(u_0,u_1)\subset B_R(0)$ for some $R>0$.
Let $p>1$ and let $u$ be a solution to \eqref{eq:DPE} in $[0,T_\varepsilon)$ according to Definition \ref{Def weak sol} and $q>-1$.
Then, for any $t\in[0,T_\varepsilon)$
\begin{align}
 \int_{\mathbb{R}^n} u(t,x)\eta_q(t,t,x)\,\mathrm{d}x
& =\varepsilon 
\int_{\mathbb{R}^n} u_0(x)\xi_q(t,0,x)\,\mathrm{d}x  + \varepsilon  t
\int_{\mathbb{R}^n} u_1(x)\eta_q(t,0,x)\,\mathrm{d}x \notag \\
&\quad + \int_0^t (t-s)
\int_{\mathbb{R}^n} |u(s,x)|^p \eta_q(t,s,x)\,\mathrm{d}x\,\mathrm{d}s .
\label{eq:prop-lower-bound}
\end{align}
\end{proposition}
\begin{proof}
We apply \eqref{WeakSol} with $\phi(s,x)= \varphi_\lambda(x)y_1(t,s;\lambda)$. By Lemma \ref{Lemma53} (iii) e (iv), it follows that
\begin{align*}
& \partial_s \phi(t,x) =-\varphi_\lambda(x), \qquad b(0)\phi(0,x)- \phi_s(0,x) = \varphi_\lambda(x)y_0(t,0;\lambda), 
\end{align*} and
\begin{align*}
& \phi_{ss}(s,x) - \Delta \phi(s,x) - \partial_s(b(s)\phi(s,x)) = 0.
 \end{align*}
Consequently,
\begin{align*}
\int_{\mathbb{R}^n} u(t,x)\varphi_\lambda(x)\,\mathrm{d}x
&= \varepsilon y_0(t,0;\lambda) \int_{\mathbb{R}^n} u_0(x)\varphi_\lambda(x)\,\mathrm{d}x + \varepsilon y_1(t,0;\lambda) \int_{\mathbb{R}^n} u_1(x)\varphi_\lambda(x)\,\mathrm{d}x \\
&\quad + \int_0^t y_1(t,s;\lambda) \left( \int_{\mathbb{R}^n} |u(s,x)|^p \varphi_\lambda(x)\,\mathrm{d}x \right)\,\mathrm{d}s.
\end{align*}
The equation in \eqref{eq:prop-lower-bound} is obtained by multiplying the previous equality by
$\lambda^q \mathrm{e}^{-\lambda(t+R)}$, then integrating with respect to $\lambda$ over
$[0,\lambda_0]$, and finally exchanging the order of integration in the variables
$\lambda$ and $x,s$. Using the definitions in \eqref{eq:xi_def} and \eqref{eq:eta_def} of
$\xi_q$ and $\eta_q$ we conclude that \eqref{eq:prop-lower-bound} is satisfied. 
\end{proof}

\begin{proposition}\label{Prop52}
Let $n\geq 2$ and $p=p_{\mathrm{Str}}(n)$. Let $u_0 \in W^{1,1}_{\mathrm{loc}}(\mathbb{R}^n)$, $u_1 \in  L^{1}_{\mathrm{loc}}(\mathbb{R}^n)$ be nonnegative, not identically zero, and compactly supported functions such that $\mathrm{supp}(u_0,u_1)\subset B_R(0)$ for some $R>0$.
Let $u$ be a solution to \eqref{eq:DPE} in $[0,T_\varepsilon)$ according to Definition \ref{Def weak sol} and let
\begin{align}\label{def functional critical case}
\mathscr{U}(t):=\int_{\mathbb{R}^n} u(t,x)\xi_q(t,t,x)\,\mathrm{d}x,
\end{align} where
\begin{align} \label{choiceofq}
    q:=\frac{n-1}{2}-\frac{1}{p}.
\end{align}
Then, there exists a positive constant $C = C(n,p, R)>0$ such that
\begin{equation}\label{iteration fram Uscr}
\mathscr{U}(t)
\ge
\frac{C}{\langle t\rangle m(t)}
\int_0^t
\frac{(t-s)\, m(s)}{\langle s \rangle \bigl(\ln\langle s \rangle\bigr)^{p-1}}
(\mathscr{U}(s))^p \,\mathrm{d}s
\end{equation}
for any $t\in[0,T_\varepsilon)$.
\end{proposition}
\begin{proof}
In this proof we adapt the approach from \cite[Proposition 4.2]{WakaYard18}. By Proposition \ref{Prop51}, being the Cauchy data, the nonlinear term and the auxiliary functions $\xi_q,\eta_q$ nonnegative, we have that $\mathscr{U}$ is a nonnegative function. 
Let $s\in [0,t]$. Applying Hölder's inequality, we deduce that
\begin{align}
\mathscr{U}(s) 
&\leq \left( \int_{\mathbb{R}^n} |u(s,x)|^p \eta_q(t,s,x)\,\mathrm{d}x \right)^{1/p} \left( \int_{|x|\leq s+R} 
\frac{\xi_q(s,s,x)^{p'}}{\eta_q(t,s,x)^{\frac{p'}{p}}}
\,\mathrm{d}x \right)^{\frac{1}{p'}}.
\label{ineq:holder_estimate}
\end{align}
By estimates (ii) and (iii) from Lemma \ref{Estiforauxiliarfunctions}, 
we can control the integral over the ball of radius $s+R$ in \eqref{ineq:holder_estimate}, up to a multiplicative constant, as follows:
\begin{align*}
     \int_{|x|\leq s+R}
\frac{\langle s \rangle^{-\frac{(n-1)p'}{2}} 
\langle s - |x| \rangle^{\left(\frac{n-3}{2}-q\right)p'}}
{\left(\frac{m(s)}{m(t)}\right)^{\frac{p'}{p}}\langle t \rangle^{-\frac{p'}{p}} \langle s \rangle^{-\frac{q p'}{p}}}
\,\mathrm{d}x
&=
 \left(\tfrac{m(t)}{m(s)}\right)^{\frac{p'}{p}} \,\langle t \rangle^{\frac{p'}{p}} 
\langle s \rangle^{\left(\frac{q}{p} - \frac{n-1}{2}\right)p'}
\int_{|x|\leq s+R}
\langle s - |x| \rangle^{\left(\frac{n-3}{2}-q\right)p'}
\,\mathrm{d}x \\
&= \left(\tfrac{m(t)}{m(s)}\right)^{\frac{p'}{p}} \,\langle t \rangle^{\frac{p'}{p}} 
\langle s \rangle^{\left(\frac{q}{p} - \frac{n-1}{2}\right)p'}
\int_{|x|\leq s+R}
\langle s - |x| \rangle^{-1}\mathrm{d}x \\
& \lesssim \left(\tfrac{m(t)}{m(s)}\right)^{\frac{p'}{p}} \,\langle t \rangle^{\frac{p'}{p}} 
\langle s \rangle^{ \left(q - \frac{(n-1)p}{2}\right)\frac{p'}{p}+n-1}  \ln \langle s \rangle.
\end{align*}
We remark that
\begin{align*}
\left(q - \tfrac{(n-1)p}{2}\right)\tfrac{p'}{p}+n-1 &= \left(\tfrac{n-1}{2}-\tfrac{1}{p} - \tfrac{(n-1)p}{2}+(n-1)\tfrac{p}{p'}\right)\tfrac{p'}{p}\\
&=\left(\tfrac{n-1}{2}-\tfrac{1}{p} - \tfrac{(n-1)p}{2}+(n-1)(p-1)\right)\tfrac{p'}{p} \\
&=\left(\tfrac{n-1}{2}(p-1)-\tfrac{1}{p}\right)\tfrac{p'}{p}
\\
&=\left(\tfrac{n-1}{2}p-\tfrac{n+1}{2}-\tfrac{1}{p}\right)\tfrac{p'}{p}+\tfrac{p'}{p}= \tfrac{p'}{p},
\end{align*} where we used \eqref{choiceofq} and $p=p_{\mathrm{Str}}(n)$. Hence, from \eqref{ineq:holder_estimate}, it follows
\[(\mathscr{U}(s))^p \lesssim \frac{m(t)}{m(s)} \langle t \rangle \langle s \rangle (\ln \langle s \rangle)^{p-1}
 \int_{\mathbb{R}^n} |u(t,x)|^p \,\eta_q(t,s,x)\,\mathrm{d}x .\]
This implies, together with Proposition \ref{Prop51} and (i) from Lemma \ref{Estiforauxiliarfunctions} the following estimate
\begin{align}
 \mathscr{U}(t)
&\ge \frac{ A_0\, \varepsilon}{m(t)} 
\int_{\mathbb{R}^n} u_0(x)\,\mathrm{d}x 
 +   \frac{B_0\, \varepsilon}{m(t)} \int_{\mathbb{R}^n} u_1(x)\,\mathrm{d}x + \frac{C}{m(t)\langle t \rangle}\int_0^t \frac{(t-s)m(s)}{\langle s \rangle (\ln \langle s \rangle)^{p-1}} (\mathscr{U}(s))^p\,\mathrm{d}s .
\end{align}
By the last inequality and the assumption on the sign of $u_0,u_1$, we conclude the validity of \eqref{iteration fram Uscr}.
\end{proof}

\begin{lemma}
Let $u_0 \in W^{1,1}_{\mathrm{loc}}(\mathbb{R}^n)$, $u_1 \in  L^{1}_{\mathrm{loc}}(\mathbb{R}^n)$ be nonnegative, not identically zero, and compactly supported functions such that $\mathrm{supp}(u_0,u_1)\subset B_R(0)$ for some $R>0$.
Let $p>1$ and let $u$ be a solution to \eqref{eq:DPE} in $[0,T_\varepsilon)$ according to Definition \ref{Def weak sol}. Then there exists a positive constant $C_0= C_0(u_0, u_1, n, p, R)$ such that
\begin{align} \label{EstiForIntup}
    \int_{\mathbb R^n} |u(t,x)|^p\,\mathrm{d}x
\ge
C_0\,\varepsilon^p\,
(m(t))^{- p}
\langle t \rangle^{\,n-1-\frac{n-1}{2}p}
\end{align}
for any $t\in [0,T_\varepsilon)$.
\end{lemma}
\begin{proof}
Let us begin by stressing that \eqref{eq:prop-lower-bound} holds for any $q>-1$. In Proposition \ref{Prop52} we fixed $q=\frac{n-1}{2}-\frac{1}{p}$, while in the proof of this lemma we are going to work with $q>\frac{n-1}{2}-\frac{1}{p}$.

Combining the lower bound estimates from Lemma \ref{Estiforauxiliarfunctions} with \eqref{eq:prop-lower-bound} and Hölder's inequality, we derive
\begin{equation}
 C_1(u_0,u_1)\,\varepsilon (m(t))^{-1}
\leq \mathscr{U}(t)
\leq 
\left( \int_{\mathbb{R}^n} |u(t,x)|^p \, \mathrm{d}x \right)^{\frac{1}{p}}
\bigl( I(t) \bigr)^{\frac{1}{p'}}.
\label{eq:5.2_rewritten}
\end{equation} where $C_1(u_0,u_1)$ is a suitable positive constant depending on some integrals involving $u_0,u_1$ and 
\[
I(t) := \int_{|x|\leq R+t} \bigl(\xi_q(t,t,x)\bigr)^{p'} \, \mathrm{d}x.
\]
Since
\[
I(t) \leq C \langle t \rangle^{\,n-1 - \frac{n-1}{2}p'},
\]
(cf. \cite[Lemma 5.1]{WakaYard18}), from \eqref{eq:5.2_rewritten} we obtain \eqref{EstiForIntup}.
\end{proof}

Now we establish the first lower bound estimates for $\mathscr{U}$.
\begin{lemma} \label{FirsSteepProce}
Let $n\geq 2$ and $p=p_{\mathrm{Str}}(n)$. Let $u_0 \in W^{1,1}_{\mathrm{loc}}(\mathbb{R}^n)$, $u_1 \in  L^{1}_{\mathrm{loc}}(\mathbb{R}^n)$ be nonnegative, not identically zero, and compactly supported functions such that $\mathrm{supp}(u_0,u_1)\subset B_R(0)$ for some $R>0$.
Let $u$ be a solution to \eqref{eq:DPE} in $[0,T_\varepsilon)$ according to Definition \ref{Def weak sol} and let $\mathscr{U}$ the functional defined in \eqref{def functional critical case}. Then, for any $t\in[ \frac32, T_\varepsilon)$ we have
\[
\mathscr{U}(t)\geq M \varepsilon^p (m(t))^{-p}
\ln\!\left(\frac{2t}{3}\right),
\]
where $M=M(u_0,u_1,n,p,R)$ is a positive constant. 
\end{lemma}
\begin{proof}
    Using (ii) from  Lemma  \ref{Estiforauxiliarfunctions} with $q=\frac{n-1}{2} -\frac{1}{p}$ and \eqref{EstiForIntup} in \eqref{eq:prop-lower-bound} we obtain:
\begin{align*}
\mathscr{U}(t) & \geq \int_0^t (t-s)\int_{\mathbb{R}^n}|u(s,x)|^p\eta_q(t,s,x) \mathrm{d}x \,\mathrm{d}s \\
& \geq \frac{B_1}{\langle t\rangle m(t)} \int_0^t (t-s)\langle s\rangle^{-q} m(s)\int_{\mathbb{R}^n}|u(s,x)|^p \mathrm{d}x \, \mathrm{d}s \\
&\geq \frac{C_0 B_1 \varepsilon^p}{\langle t \rangle m(t)}
\int_0^t (t-s) (m(s))^{1-p}\langle s \rangle^{-q + (n-1)(1-\frac{p}{2})} \, \mathrm{d}s =   \frac{C_0 B_1 \varepsilon^p}{(m(t))^{p}\langle t \rangle}
\int_0^t \frac{t - s}{\langle s \rangle} \, \mathrm{d}s,
\end{align*} where we used $p=p_{\mathrm{Str}}(n)$ in the last step.
Finally, for $t\in[\frac{3}{2},T_\varepsilon)$, since $\langle s\rangle \leq 3s$ for any $s\geq \frac{3}{2}$, integrating by parts we obtain
\begin{align*}
\mathscr{U}(t) &\geq  \frac{C_0 B_1 \varepsilon^p}{9 t\, (m(t))^{p}}
\int_1^t \frac{t - s}{s} \, \mathrm{d}s \\
&\geq  \frac{C_0 B_1 \varepsilon^p}{9 t \,(m(t))^{p}}
\int_1^t \ln s \, \mathrm{d}s \\
&\geq  \frac{C_0 B_1 \varepsilon^p}{9 t\, (m(t))^{p}}
\int_{\frac{2t}{3}}^t \ln s \, \mathrm{d}s =
\frac{C_0 B_1 \varepsilon^p}{27\,(m(t))^{p}}
\ln\!\left(\frac{2t}{3}\right).
\end{align*} Setting $M:= \frac{B_1C_0}{27}$ we completed the proof.
\end{proof}

\subsection{Iteration argument}
In this subsection, we complete the proof of Theorem \ref{Thm:CriticalCase}. 

Hereafter, we work under the following assumptions (when $n\geq 2$ and $p=p_{\mathrm{Str}}(n)$) for the Cauchy data: $u_0 \in H^1(\mathbb{R}^n)$, $u_1 \in L^2(\mathbb{R}^n)$ are nonnegative, not identically zero, and compactly supported functions such that $\mathrm{supp}(u_0,u_1)\subset B_R(0)$ for some $R>0$. Let $u$ be a solution to \eqref{eq:DPE} in $[0,T_\varepsilon)$ according to Definition \ref{Def weak sol}.
By $\mathscr{U}$ we denote the functional defined in \eqref{def functional critical case} with $q$ given by \eqref{choiceofq}.

We will prove that $\mathscr{U}$ satisfies the following sequence of lower bound estimates:
\begin{equation}
\mathscr{U}(t)\ge C_j \varepsilon^{p^{j+1}}\,(\ln\langle t\rangle)^{-b_j}\, (m(t))^{-p^{j+1}}
\left(\ln\left(\frac{t}{\ell_j}\right)\right)^{a_j} \qquad \mbox{for any} \ j\in\mathbb{N} \mbox{ and any } t\in [\ell_j,T_\varepsilon)
\label{eq:iter-new}
\end{equation} where $\ell_j:=2-2^{-(j+1)}$ for any $j\in\mathbb{N}$ and $\{a_j\}_{j\in\mathbb{N}}$, $\{b_j\}_{j\in\mathbb{N}}$, $\{C_j\}_{j\in\mathbb{N}}$ are sequences of nonnegative numbers to be determined iteratively.

We have already proved the case $j=0$ in Lemma \ref{FirsSteepProce} with $C_0:=M$, $a_0:=1$ and $b_0:=0$.
Next, assume that \eqref{eq:iter-new} is satisfied for some $j \ge 1$ and we prove it for $j+1$.

Combining \eqref{eq:iter-new} and the iteration frame in \eqref{iteration fram Uscr} from Proposition \ref{Prop52}, we obtain for $t \in [\ell_{j+1},T_\varepsilon)$
\begin{align*}
\mathscr{U}(t) & \geq 
\frac{C}{\langle t\rangle m(t)}
\int_{\ell_j}^t
\frac{(t-s)\, m(s)}{\langle s \rangle \bigl(\ln\langle s \rangle\bigr)^{p-1}}
(\mathscr{U}(s))^p \,\mathrm{d}s \\
& \geq 
\frac{CC_j^p \varepsilon^{p^{j+2}}}{\langle t\rangle m(t)}
\int_{\ell_j}^t
\frac{t-s}{\langle s \rangle} \bigl(\ln\langle s \rangle\bigr)^{-pb_j-(p-1)}  (m(s))^{-p^{j+2}+1} \left(\ln\left(\frac{s}{\ell_j}\right)\right)^{pa_j} \,\mathrm{d}s
\end{align*}
By using the monotonicity of the multiplier $m$ and the nonnegativity of $b_j$, for $t \in [\ell_{j+1},T_\varepsilon)$ we have
\begin{align}
\mathscr{U}(t) & \geq 
\frac{CC_j^p \varepsilon^{p^{j+2}}}{\langle t\rangle} (m(t))^{-p^{j+2}}
\bigl(\ln\langle t \rangle\bigr)^{-pb_j-(p-1)}   \int_{\ell_j}^t
\frac{t-s}{\langle s \rangle}  \left(\ln\left(\frac{s}{\ell_j}\right)\right)^{pa_j} \,\mathrm{d}s \notag \\
& \geq 
\frac{CC_j^p \varepsilon^{p^{j+2}}}{9} (m(t))^{-p^{j+2}} \bigl(\ln\langle t \rangle\bigr)^{-pb_j-(p-1)}   
\frac{1}{t} \int_{\ell_j}^t
\frac{t-s}{t} \left(\ln\left(\frac{s}{\ell_j}\right)\right)^{pa_j} \,\mathrm{d}s. \label{IF est 1}
\end{align}
For $t \in [\ell_{j+1},T_\varepsilon)$, after integrating by parts in the last integral, we can shrink the domain of integration from $[\ell_j,t]$ to $\left[\ell_j t/\ell_{j+1},t\right]$, obtaining
\begin{align}
\frac{1}{t}\int_{\ell_j}^t
\frac{t-s}{t} \left(\ln\left(\frac{s}{\ell_j}\right)\right)^{pa_j} \,\mathrm{d}s & = \frac{1}{t}\left\{\left[\frac{t-s}{pa_j+1}\left(\ln\left(\frac{s}{\ell_j}\right)\right)^{pa_j+1}\right]^{s=t}_{s=\ell_j}+\int_{\ell_j}^t \frac{1}{pa_j+1} \left(\ln\left(\frac{s}{\ell_j}\right)\right)^{pa_j+1} \mathrm{d}s \right\} \notag \\
& = (pa_j+1)^{-1}\, \frac{1}{t} \int_{\ell_j}^t  \left(\ln\left(\frac{s}{\ell_j}\right)\right)^{pa_j+1} \mathrm{d}s \notag \\
& \geq (pa_j+1)^{-1}\, \frac{1}{t} \int_{\tfrac{\ell_j t}{\ell_{j+1}}}^t  \left(\ln\left(\frac{s}{\ell_j}\right)\right)^{pa_j+1} \mathrm{d}s \notag\\
& \geq (pa_j+1)^{-1}\, \frac{1}{t} \left(\ln\left(\frac{t}{\ell_{j+1}}\right)\right)^{pa_j+1} \left(1-\frac{\ell_j}{\ell_{j+1}}\right)t \notag \\
& \geq (pa_j+1)^{-1}\, 2^{-(j+3)} \left(\ln\left(\frac{t}{\ell_{j+1}}\right)\right)^{pa_j+1} , \label{IF est 2}
\end{align} where we used the inequality $1-\ell_j/\ell_{j+1}\geq 2^{-(j+3)}$ in the last step. By \eqref{IF est 1} and \eqref{IF est 2}, we conclude that for  $t \in [\ell_{j+1},T_\varepsilon)$ it holds
\begin{align*}
\mathscr{U}(t) & \geq 
\frac{CC_j^p }{9\cdot 2^{j+3} (pa_j+1) }   
 \, \varepsilon^{p^{j+2}} (m(t))^{-p^{j+2}} \bigl(\ln\langle t \rangle\bigr)^{-pb_j-(p-1)}  \left(\ln\left(\frac{t}{\ell_{j+1}}\right)\right)^{pa_j+1}, 
\end{align*} which is exactly \eqref{eq:iter-new} for $j+1$ provided that
\begin{align*}
a_{j+1}:=pa_j+1, \quad  b_{j+1}:=pb_j+p-1, \quad C_{j+1}:= \frac{CC_j^p }{9\cdot 2^{j+3} a_{j+1} }.
\end{align*}
By using the previous recursive relations for $a_j$ and $b_j$, we obtain
\begin{align*}
a_j & = pa_{j-1}+1= p^2a_{j-2}+p+1= \ldots = p^j a_0+\sum_{k=0}^{j-1}p^k=\frac{p^{j+1}-1}{p-1},\\
b_j & = pb_{j-1}+p-1= p^2b_{j-2}+(p-1)(p+1)= \ldots = p^j b_0+(p-1)\sum_{k=0}^{j-1}p^k=p^{j}-1,
\end{align*}
since $a_0=1$ and $b_0=0$. Hence, $C_j\geq D (2p)^{-j} C_{j-1}^p$, where $D:=C(p-1)/(9p)$. Applying the logarithmic functions to both sides of the previous inequality for $C_j$ and using iteratively the resulting inequality, we find
\begin{align*}
\ln C_j & \geq p \ln C_{j-1}-j \ln(2p) + \ln D \\
& \geq p^2 \ln C_{j-2} -(j+(j-1)p)\ln (2p)+(1+p) \ln D \\
& \geq \ldots \geq  p^j \ln C_{0} -\ln (2p) \sum_{k=0}^{j-1}(j-k)p^k+\ln D \sum_{k=0}^{j-1}p^k \\
& =  p^j \ln C_{0} -\frac{\ln (2p)}{p-1} \left(\frac{p^{j+1}-p}{p-1}-j\right)+\ln D \, \frac{p^j-1}{p-1} \\
&= p^j \ln E+\frac{\ln(2p)}{p-1}j+\frac{p\ln(2p)}{(p-1)^2}-\frac{\ln D}{p-1},
\end{align*} where $E:= C_0 (2p)^{-p/(p-1)^2} D^{1/(p-1)}$. Let us denote $j_1:=\lfloor \frac{\ln D}{\ln(2p)}-\frac{p}{p-1} \rfloor_+$. Then, for any $j\geq j_1$
\begin{align}
\ln C_j \geq p^j \ln E. \label{ln Cj lb}
\end{align}
Combining \eqref{eq:iter-new}, the representations for $a_j$, $b_j$ and \eqref{ln Cj lb}, since $\ell_j \uparrow 2$ as $j\to +\infty$, for $t\in[2,T)$ and $j\geq j_1$ we obtain
\begin{align*}
\mathscr{U}(t) &\geq C_j \varepsilon^{p^{j+1}} (m(t))^{-p^{j+1}} \left(\ln\left(\frac{t}{2}\right)\right)^{a_j}\left(\ln\langle t\rangle\right)^{-b_j} \\
&= C_j \varepsilon^{p^{j+1}} (m(t))^{-p^{j+1}} \left(\ln\left(\frac{t}{2}\right)\right)^{\frac{p^{j+1}-1}{p-1}}\left(\ln\langle t\rangle\right)^{-p^j+1} \\
& \geq  \exp \left\{  p^j\left[ \ln (E\varepsilon^{p})-p\ln (m(t))+\frac{p}{p-1}\ln\left(\ln\left(\frac{t}{2}\right)\right)-\ln(\ln\langle t \rangle)\right]\right\} \left(\ln\left(\frac{t}{2}\right)\right)^{-1/(p-1)}\left(\ln\langle t\rangle\right).
\end{align*} We remark that for $t\geq 4$: $\ln\left(\frac{t}{2}\right)\geq \frac{\ln t}{2}$ and $\ln\langle t \rangle\leq 2 \ln t$. Therefore, for $j\geq j_1$ and $t\in[4,T_\varepsilon)$ we proved that
\begin{align}
\mathscr{U}(t) &\geq \exp \left\{  p^j\left[ \ln (E\varepsilon^{p})-p\ln (m(t))+\tfrac{p}{p-1}\ln\left(\tfrac{1}{2}\ln t\right)-\ln(2\ln t)\right]\right\} \left(\ln\left(\frac{t}{2}\right)\right)^{-1/(p-1)}\left(\ln\langle t\rangle\right)\notag \\
& =  \exp \left\{  p^j\left[ \ln (H\varepsilon^{p})-p\ln (m(t))+\frac{\ln\left(\ln t\right)}{p-1}\right]\right\} \left(\ln\left(\frac{t}{2}\right)\right)^{-1/(p-1)}\left(\ln\langle t\rangle\right), \label{pre key}
\end{align} where $H:= 2^{-\frac{2p-1}{p-1}}E$.
Let us denote 
\begin{align*}
K(t,\varepsilon) &:= \ln(H \varepsilon^p)-p\ln(m(t))+\frac{\ln(\ln t))}{p-1}\\ & = \ln \left(H\varepsilon^p (m(t))^{-p}(\ln t)^{1/(p-1)}\right) \\  & = \ln \left(H\varepsilon^p (\ln^{[k]}(\mathrm{e}^{[k]}+t))^{-\mu p}(\ln t)^{1/(p-1)}\right),
\end{align*} then, for $j\geq j_1$ and $t\in[4,T_\varepsilon)$, we can rewrite  \eqref{pre key} as follows:
\begin{align}\label{key}
\mathscr{U}(t) &\geq  \exp \left\{  p^j K(t,\varepsilon)\right\} \left(\ln\left(\tfrac{t}{2}\right)\right)^{-1/(p-1)}\left(\ln\langle t\rangle\right).
\end{align}
It is clear that if $t\geq 4$ is such that $K(t,\varepsilon)>0$, then $\mathscr{U}(t)$ cannot be finite, since its lower bound in \eqref{key} blows up as $j\to +\infty$. Therefore, we investigate now under which conditions requirements on $\varepsilon$ and $t$ we have $K(t,\varepsilon)>0$. We have to consider separately the case $k\geq 2$ from the case $k=1$.

For $k\geq 2$, we introduce the function $$g(t):= (m(t))^{-p}(\ln t)^{1/(p-1)}=  (\ln^{[k]}(\mathrm{e}^{[k]}+t))^{-\mu p}(\ln t)^{1/(p-1)}.$$
Since
\begin{align*}
g'(t) &=-p (m(t))^{-p-1}\, m'(t) (\ln t)^{1/(p-1)}+(m(t))^{-p}\, \frac{ (\ln t)^{1/(p-1)-1}}{(p-1) t} \\
&=-p (m(t))^{-p}\, b(t) (\ln t)^{1/(p-1)}+(m(t))^{-p}\,  \frac{ (\ln t)^{1/(p-1)-1}}{(p-1) t}  \\
&=(m(t))^{-p} (\ln t)^{1/(p-1)} \left[-pb(t)+\frac{1}{(p-1)t \ln t}\right],
\end{align*} we have that $g'(t)>0$ for $t$ sufficiently large, moreover $\displaystyle{\lim_{t\to +\infty}g(t)=+\infty}$. Consequently, there exists $t_1>0$ such that $g:[t_1,+\infty)\to [g(t_1),+\infty)$ is an increasing bijection. Hence $K(t,\varepsilon)>0$ if and only if $g(t)>H^{-1}\varepsilon^{-p}$. We fix $\varepsilon_0=\varepsilon_0(n,u_0,u_1,\mu,k,R)>0$ such that $H^{-1}\varepsilon_0^{-p}\geq g(t_1)$ and $g^{-1}(H^{-1}\varepsilon_0^{-p})\geq 4$. Then, for any $\varepsilon\in (0,\varepsilon_0]$ such that $t\geq g^{-1}(H^{-1}\varepsilon^{-p})$ we have $t\geq \max\{4,t_1\}$ and $K(t,\varepsilon)>0$. From the previous considerations, we conclude that $\mathscr{U}$ blows up in finite time and the following upper bound estimate for the lifespan $T_\varepsilon \leq g^{-1}(H^{-1}\varepsilon^{-p})$, which is equivalent to \eqref{Upper bound lifespan crit k>1}.

For $k=1$, we derive another lower bound for $\mathscr{U}$, written in a more convenient way. By \eqref{pre key}, for $t\in[4,T_\varepsilon)$ and $j\geq j_1$ we have
\begin{align}
\mathscr{U}(t) &\geq \exp \left\{  p^j\left[ \ln (H\varepsilon^{p})-p\ln (\ln(\mathrm{e}+t))^{\mu}+\frac{\ln\left(\ln t\right)}{p-1}\right]\right\} \left(\ln\left(\frac{t}{2}\right)\right)^{-1/(p-1)}\left(\ln\langle t\rangle\right) \notag\\
 &\geq \exp \left\{  p^j\left[ \ln (H\varepsilon^{p})-\mu p (\ln\ln t+\ln 2)+\frac{\ln\left(\ln t\right)}{p-1}\right]\right\} \left(\ln\left(\frac{t}{2}\right)\right)^{-1/(p-1)}\left(\ln\langle t\rangle\right) \notag \\
  & = \exp \left\{  p^j\left[ \ln (N\varepsilon^{p})+\left(\frac{1}{p-1}-\mu p\right)\ln \ln t\right]\right\} \left(\ln\left(\frac{t}{2}\right)\right)^{-1/(p-1)}\left(\ln\langle t\rangle\right), \label{key2}
\end{align}
where $N:=2^{-\mu p}H$ and we used the inequality $\ln(\mathrm{e}+t)\leq 2\ln t$ for $t\geq 4$. Let us define 
\begin{align*}
L(t,\varepsilon) & := \ln(N\varepsilon^p)+\left(\frac{1}{p-1}-\mu p\right)\ln \ln t  \\
&=\ln \left(N\varepsilon^p (\ln t)^{1/(p-1)-\mu p}\right)
\end{align*} From the previous definition it follows the restriction on $\mu$ in the case $k=1$: in order to be able to prove a blow-up result, we need to require that $t\geq 4$ and $L(t,\varepsilon)>0$. Since the condition $L(t,\varepsilon)>0$ has to be fulfilled for $t$ large and $\varepsilon$ small, necessarily we have to require that $1/(p-1)-\mu p>0$. Let us point out that $1/(p-1)-\mu p>0$ corresponds exactly to the condition $\mu<\mu_0^*(n)$ in the statement of Theorem \ref{Thm:CriticalCase}.

Hence, for $\mu\in(0,\mu_0^*(n))$, we find that the condition $L(t,\varepsilon)>0$ holds if and only if $t>\exp\left(\widetilde{N}\varepsilon^{-\frac{p(p-1)}{1-\mu p(p-1)}}\right)$, where $\widetilde{N}:=N^{-\frac{p-1}{1-\mu p(p-1)}}$. Therefore, for $k=1$ and $\mu\in(0,\mu_0^*(n))$ we can fix $\varepsilon_0=\varepsilon_0(n,\mu,u_0,u_1,R)>0$ such that $\exp\Big(\widetilde{N}\varepsilon_0^{-\frac{p(p-1)}{1-\mu p(p-1)}}\Big)\geq 4$. Then, for any $\varepsilon\in(0,\varepsilon_0]$ and $t>\exp\left(\widetilde{N}\varepsilon^{-\frac{p(p-1)}{1-\mu p(p-1)}}\right)$ we have that $t\geq 4$ and $L(t,\varepsilon)>0$, so letting $j\to+\infty$ in \eqref{key2} we obtain the blow-up in finite time also in this case and the upper bound estimate for the lifespan in \eqref{Upper bound lifespan crit k=1}.

\section{Blow-up result for the derivative type nonlinearity}
\label{Section sub Glassey}

In this section, we prove Theorems \ref{Thm subcritical Glassey} and \ref{Thm critical Glassey}.
We follow the approach from \cite[Section 13.2]{LZ2017}  and \cite{LaiTakamura2019}.  
\\ As in the previous sections, we denote by $\varphi$ and $\Phi$ the functions defined in  in \eqref{Espe varphi} and \eqref{EspePsi}, respectively. From Lemma \ref{Lemma estimate Phi in Lp'}, we have that
\begin{equation}\label{eq:psiint}
 \int_{|x|\le t+R}\Phi(t,x)\, \mathrm{d}x\lesssim (1+t)^{(n-1)/2}
\end{equation} for any $t \geq 0$.

Let us assume that $u_0,u_1$ satisfy the assumptions of Theorems \ref{Thm subcritical Glassey}, \ref{Thm critical Glassey} and let $u$ be a weak solution to \eqref{eq:DPE der} according to Definition \ref{Def weak sol}.
\\ We introduce the following time-dependent functions associated with $u$
\begin{align}
 U_0(t)& :=\int_{\R^n}u(t,x)\Phi(t,x)\,\mathrm{d} x, \label{eq:F1} \\
 U_1(t)& :=\int_{\R^n}u_t(t,x)\Phi(t,x)\,\mathrm{d} x, \label{eq:U1} 
\end{align}
 for any $t \in [0, T_\varepsilon)$.

Now we define the time-dependent function, that we will use to prove the blow-up in finite time of $u$:
\begin{equation}\label{eq:H}
 H(t):=\frac{\varepsilon}{2}\int_{\R^n}u_1(x)\varphi(x)\,\mathrm{d} x+\frac12\int_0^t m(s)\int_{\R^n}|u_t(s,x)|^p\Phi(s,x)\,\mathrm{d} x\,\mathrm{d} s
 ,
\end{equation} where $m$ is defined in \eqref{multiplier}. \\
Since $H(0)>0$ and $H'(t)\ge 0$ for any $t\in[0,T_\varepsilon)$, the function $H$ is nondecreasing on $[0,T_\varepsilon)$. \\ In particular, $H(t)\ge H(0)>0$ for any $t\in[0,T_\varepsilon)$.

In the next lemma, we obtain the ordinary differential inequality for $H$ that provides the blow-up of $u$ and the corresponding upper bound estimates for the lifespan.
\begin{lemma}\label{lem:diffineq}
There exists $C = C (n,p,R,\mu,k)>0$ such that
\begin{equation}\label{eq:Hdiff}
 H'(t)\ge C (1+t)^{-\frac{n-1}{2} (p-1)}
 (m(t))^{-(p-1)} (H(t))^p,
\end{equation} for any $t\in[0,T_\varepsilon)$.
\end{lemma}
\begin{proof}
Taking $\Phi$ as a test function in \eqref{SolDef} and differentiating with respect to the time variable the resulting identity, we obtain for any $t\in[0,T_\varepsilon)$
\begin{align}\label{eq:basic-id}
\frac{\mathrm{d}}{\mathrm{d}t}\int_{\R^n}u_t(t,x)\Phi(t,x) \,\mathrm{d} x
-\int_{\R^n}u_t(t,x)\Phi_t(t,x)\,\mathrm{d} x
+\int_{\R^n}\nabla u(t,x)\cdot\nabla\Phi (t,x)\,\mathrm{d} x \nonumber \\
+ \,b(t)\int_{\R^n}u_t(t,x)\Phi(t,x)\,\mathrm{d} x
=\int_{\R^n}|u_t(t,x)|^p\Phi(t,x)\,\mathrm{d} x . 
\end{align}
Since $\Phi_t=-\Phi$ and $\Delta \Phi=\Phi$, from \eqref{eq:basic-id} we have for any $t\in[0,T_\varepsilon)$
\begin{align*}
& \frac{\mathrm{d}}{\mathrm{d}t}\int_{\R^n} u_t(t,x)\Phi(t,x)\,\mathrm{d} x+b(t)\int_{\R^n} u_t(t,x)\Phi(t,x)\,\mathrm{d} x+
\int (u_t(t,x)-u(t,x))\Phi(t,x)\,\mathrm{d} x \\ & \qquad =
\int_{\R^n} |u_t(t,x)|^p\Phi(t,x)\,\mathrm{d} x,
\end{align*} which can be rewritten as
\begin{align}\label{eq:dutpsi U0 U1}
U_1'(t)+b(t)U_1(t)+U_1(t) -U_0(t)
= 
\int_{\R^n} |u_t(t,x)|^p\Phi(t,x)\,\mathrm{d} x.
\end{align}
Multiplying \eqref{eq:dutpsi U0 U1} by $m$, for any $t\in[0,T_\varepsilon)$ we find
\begin{equation}\label{eq:idA}
\frac{\mathrm{d}}{\mathrm{d}t}\left(m(t)U_1(t)\right)
+m(t)\left(U_1(t)-U_0(t)\right)
=m(t)\int_{\R^n} |u_t(t,x)|^p\Phi(t,x)\,\mathrm{d} x.
\end{equation}
Since 
\begin{align*}
\frac{\mathrm{d}}{\mathrm{d}t}\left(m(t)U_0(t)\right) & =m(t)U'_0(t)+m(t)b(t)U_0(t) \\ & =m(t)\left(U_1(t)-U_0(t)\right)+m(t)b(t)U_0(t),
\end{align*}
from \eqref{eq:idA} we have for any $t\in[0,T_\varepsilon)$ 
\begin{equation*}
\frac{\mathrm{d}}{\mathrm{d}t}\left(m(t)\left(U_1(t)+U_0(t)\right)\right)
=m(t)\int_{\R^n} |u_t(t,x)|^p\Phi(t,x)\,\mathrm{d} x+b(t)m(t)U_0(t).
\end{equation*}
By Lemma \ref{Lemma U0 lb}, the function $U_0$  is nonnegative, therefore, for any $t\in[0,T_\varepsilon)$
\begin{equation}\label{eq:idB}
\frac{\mathrm{d}}{\mathrm{d}t}\left(m(t)(U_1(t)+U_0(t))\right)
\geq m(t)\int_{\R^n} |u_t(t,x)|^p\Phi(t,x)\,\mathrm{d} x.
\end{equation}
 Integrating \eqref{eq:idB} over $[0,t]$, since $m(0)=1$, for any $t\in[0,T_\varepsilon)$ we get
\begin{equation}\label{eq:intB}
 m(t)\left(U_1(t)+U_0(t)\right)
\ge \varepsilon \int (u_0(x)+u_1(x))\varphi(x)\,\mathrm{d} x+
\int_0^t m(s)\int_{\R^n} |u_t(s,x)|^p\Phi(s,x)\, \mathrm{d} x \,\mathrm{d} s .
\end{equation}
Adding \eqref{eq:intB} to \eqref{eq:idA} gives
\begin{align}\label{eq:Gineq}
\frac{\mathrm{d}}{\mathrm{d}t}\left(m(t)U_1(t)\right)+2m(t)U_1(t) & \ge \varepsilon \int_{\R^n}(u_0(x)+u_1(x))\varphi(x)\,\mathrm{d} x +m(t)\int_{\R^n} |u_t(t,x)|^p\Phi(t,x)\,\mathrm{d} x \notag \\ & \qquad+
\int_0^t m(s)\int_{\R^n} |u_t(s,x)|^p\Phi(s,x)\,\mathrm{d} x \,\mathrm{d} s,
\end{align} for any $t\in[0,T_\varepsilon)$. \\
Set $G(t):=m(t)U_1(t)- H(t)$.
Then $G(0)=\frac{\varepsilon}{2}\int_{\mathbb{R}^n}u_1(x)\varphi(x) \,\mathrm{d}x>0$, and \eqref{eq:Gineq} implies
\begin{align*}
G'(t)+2G(t) &= \frac{\mathrm{d}}{\mathrm{d}t}\left(m(t)U_1(t)\right)+2m(t)U_1(t) - (H'(t)+2H(t)) \\ &\ge \frac{m(t)}{2}\int_{\R^n} |u_t(t,x)|^p\Phi(t,x) \,\mathrm{d} x+
\varepsilon  \int_{\R^n} u_0(x)\varphi(x)\,\mathrm{d} x \\ &\ge0
\end{align*} for any $t\in[0,T_\varepsilon)$. \\
Thus, $G(t)\ge \e^{-2t}G(0)>0$, that is, 
\begin{equation}\label{eq:lowerH}
m(t) U_1(t) \ge H(t)
\end{equation} for any $t\in[0,T_\varepsilon)$ and, in particular, $U_1$ is nonnegative. \\
By Hölder's inequality and \eqref{eq:psiint},
\begin{align}
(U_1(t))^p
&\le \left(\int_{|x|\le t+R}\Phi(t,x)\mathrm{d} x\right)^{p-1} \int_{\R^n} |u_t(t,x)|^p\Phi(t,x)\, \mathrm{d} x
\notag \\
&\lesssim  (1+t)^{\frac{n-1}{2} (p-1)}
\int |u_t(t,x)|^p\Phi(t,x)\, \mathrm{d} x \notag \\
& = \frac{2}{m(t)}  (1+t)^{\frac{n-1}{2} (p-1)} H'(t). \label{ineq U1^p H'}
\end{align}
Combining \eqref{eq:lowerH} and \eqref{ineq U1^p H'}, it follows that
\begin{align*}
H'(t) & \gtrsim  (1+t)^{-\frac{n-1}{2} (p-1)} m(t) (U_1(t))^p \\
 & \gtrsim  (1+t)^{-\frac{n-1}{2} (p-1)} (m(t))^{1-p} (H(t))^p,
\end{align*} which is \eqref{eq:Hdiff}.
\end{proof}

Finally, by using a comparison argument for $H$, we complete the proofs of Theorems \ref{Thm subcritical Glassey} and \ref{Thm critical Glassey}.

\begin{proof}[Proof of Theorem \ref{Thm subcritical Glassey}]
From \eqref{eq:Hdiff} we have
\begin{align*}
\frac{H'(t)}{(H(t))^p}=\frac{\mathrm{d}}{\mathrm{d}t}\left(\frac{(H(t))^{1-p}}{1-p}\right)\geq C (1+t)^{-\frac{n-1}{2}(p-1)}(m(t))^{1-p}
\end{align*} for any $t\in[0,T_\varepsilon)$. Hence, integrating over $[0,t]$ we get
\begin{align}\label{int ineq H}
(H(0))^{1-p}-(H(t))^{1-p} \geq C(p-1) \int_0^t (1+s)^{-\frac{n-1}{2}(p-1)}(m(s))^{1-p}\, \mathrm{d}s
\end{align} for any $t\in[0,T_\varepsilon)$. By the monotonicity of $m$ and $p>1$ it follows that 
\begin{align*}
(H(0))^{1-p}-(H(t))^{1-p} \geq C(p-1) (m(t))^{1-p} \int_0^t (1+s)^{-\frac{n-1}{2}(p-1)}\, \mathrm{d}s
\end{align*} and, hence,
\begin{align} \label{int lb H sub Gla}
\left((H(0))^{1-p} -C(p-1) (m(t))^{1-p} \int_0^t (1+s)^{-\frac{n-1}{2}(p-1)}\, \mathrm{d}s \right)^{-\frac{1}{p-1}} \leq H(t) 
\end{align} for any $t\in[0,T_\varepsilon)$. \\
From \eqref{int lb H sub Gla} we obtain that $H$ blows up in finite time, since $1-\frac{n-1}{2}(p-1)>0$. \\ Let us derive now the upper bound estimate for the lifespan.
 By \eqref{int lb H sub Gla} we see that $H(t)$ is not finite when the base of the $-1/(p-1)$ power on the left-hand side of \eqref{int lb H sub Gla} is non-positive, i.e.
\begin{align*}
\widetilde{C}\varepsilon^{-(p-1)}\leq (m(t))^{1-p} \int_0^t (1+s)^{-\frac{n-1}{2}(p-1)}\, \mathrm{d}s,
\end{align*} where $\widetilde{C}:= \frac{2^{p-1}}{C(p-1)}\left(\int_{\mathbb{R}^n}u_1(x) \varphi(x) \, \mathrm{d}x \right)^{1-p}$. Consequently, since
\begin{align*}
 \int_0^t (1+s)^{-\frac{n-1}{2}(p-1)}\, \mathrm{d}s \leq \int_0^t s^{-\frac{n-1}{2}(p-1)}\, \mathrm{d}s= \frac{t^{1-\frac{n-1}{2}(p-1)}}{1-\frac{n-1}{2}(p-1)}
\end{align*}
we conclude that $H(t)$ cannot be finite for $t$ such that 
\begin{align*}
\widehat{C}\varepsilon^{-(p-1)}\leq (m(t))^{1-p} t^{1-\frac{n-1}{2}(p-1)},
\end{align*} where $\widehat{C}:=\widetilde{C}(1-\frac{n-1}{2}(p-1))$.
Therefore, we proved \eqref{lifespan estimate subcritical case Gla}.
\end{proof}

\begin{proof}[Proof of Theorem \ref{Thm critical Glassey}]
Let us begin with the case $k\geq 2$. From \eqref{int lb H sub Gla}, we have that $H$ cannot be globally in time defined even when $1-\frac{n-1}{2}(p-1)=0$. We derive now the upper bound for the lifespan. Let $t\geq 1$, then, since $p=p_{\mathrm{Gla}}(n)$, from \eqref{int lb H sub Gla} we have
\begin{align*}
 H(t) & \geq \left((H(0))^{1-p} -C(p-1) (m(t))^{1-p} \int_0^{t-1} (1+s)^{-\frac{n-1}{2}(p-1)}\, \mathrm{d}s \right)^{-\frac{1}{p-1}} \\
& = \left((H(0))^{1-p} -C(p-1) (m(t))^{1-p} \ln t \right)^{-\frac{1}{p-1}}.
\end{align*} Proceeding similarly as in the previous proof, we obtain \eqref{Upper bound lifespan crit k>1 Gla}. \\
Let us consider now the case $k=1$. \\
When $\mu\in(0,\frac{1}{p-1})$, from \eqref{int ineq H} we get for $t\geq 1$
\begin{align*}
(H(0))^{1-p}-(H(t))^{1-p} & \gtrsim \int_0^t (1+s)^{-1}(\ln(\mathrm{e}+s))^{-\mu(p-1)}\, \mathrm{d}s \\
 & \gtrsim \int_0^{t-1} (1+s)^{-1}(\ln(1+s))^{-\mu(p-1)}\, \mathrm{d}s \approx (\ln t)^{1-\mu(p-1)}
\end{align*} and, consequently,
\begin{align}\label{lb H crit mu<mu1}
H(t)\geq \left((H(0))^{1-p}-C_1 (\ln t)^{1-\mu(p-1)}\right)^{-\frac{1}{p-1}}, 
\end{align} for some $C_1>0$. \\ On the other hand, when $\mu=\frac{1}{p-1}=\frac{n-1}{2}$, by \eqref{int ineq H} we obtain for $t\geq \mathrm{e}$
\begin{align*}
(H(0))^{1-p}-(H(t))^{1-p} & \geq C(p-1) \int_0^t (1+s)^{-1}(\ln(\mathrm{e}+s))^{-1}\, \mathrm{d}s \\
 & \geq C(p-1) \int_0^{t-\mathrm{e}} (\mathrm{e}+s)^{-1}(\ln(\mathrm{e}+s))^{-1}\, \mathrm{d}s = C(p-1) \ln\ln t
\end{align*} and, therefore,
\begin{align}\label{lb H crit mu=mu1}
H(t)\geq \left((H(0))^{1-p}-C(p-1) \ln\ln t \right)^{-\frac{1}{p-1}}. 
\end{align} 
By \eqref{lb H crit mu<mu1} and \eqref{lb H crit mu=mu1} we conclude that $H$ blows up in finite time and that \eqref{Upper bound lifespan crit k=1 Gla} holds.
\end{proof}

\begin{remark}
After completing the preparation of the present paper, we became aware of
the recent preprint \cite{CH26}, where results overlapping with ours for
problem \eqref{eq:DPE der} are obtained.
\end{remark}

  \section*{Acknowledgments}
 
W. Nunes do Nascimento  was financially supported by the TNE-DESK project during his visit to the University of Bari, where part of this work was developed. \\
A. Palmieri is partially supported by the PRIN 2022 project “Anomalies in partial differential equations and applications” CUP H53C24000820006. \\
A. Palmieri is a member of the \emph{Gruppo Nazionale per L’Analisi Matematica, la Probabilità e le loro Applicazioni} (GNAMPA) of the \emph{Instituto Nazionale di Alta Matematica} (INdAM).


\end{document}